\input ./style/arxiv-general.cfg
\documentclass[aap,MSNbibl,nameyear,seceqn,dvips]{arximspdf}
\makeatletter
   \@ifpackageloaded{graphicx}{}{\usepackage{graphicx}}
\makeatother
\usepackage{mathbh}

\doi{10.1214/15-AAP1101}
\volume{26}
\issue{2}
\pubyear{2016}
\firstpage{722}
\lastpage{759}
\docsubty{FLA}

\makeatletter
\newcommand{\eqref}[1]{(\ref{#1})}
\newtheorem{proposition}{Proposition}[section]
\newtheorem{lemma}[proposition]{Lemma}
\newtheorem{lemmaa}{Lemma}[section]
\newtheorem{theorem}[proposition]{Theorem}
\newproclaim{definition}[proposition]{Definition}
\newtheorem{corollary}[proposition]{Corollary}
\newproclaim{remark}[proposition]{Remark}
\newproclaim{example}[proposition]{Example}
\def\La{\Lambda}
\def\ep{\varepsilon}

\newcommand{\wt}{\widetilde}
\def\R{{\mathbb R}}
\def\E{{\mathbb E}}
\def\P{{\mathbb P}}

\makeatother

\begin{document}
\begin{frontmatter}

\title{The mean Euler characteristic and excursion probability of
Gaussian random fields with stationary increments\thanksref{T1}}
\thankstext{T1}{Supported in part by NSF Grants DMS-10-06903,
DMS-13-07470 and DMS-13-09856.}
\runtitle{The mean Euler characteristic and excursion probability}

\begin{aug}
\author[A]{\fnms{Dan}~\snm{Cheng}\ead[label=e1]{cheng@stt.msu.edu}}
\and
\author[B]{\fnms{Yimin}~\snm{Xiao}\corref{}\ead[label=e2]{xiao@stt.msu.edu}}
\runauthor{D. Cheng and Y. Xiao}
\affiliation{North Carolina State University and Michigan State University}
\address[A]{Department of Statistics\\
North Carolina State University\\
2311 Stinson Drive\\
Campus Box 8203\\
Raleigh, North Carolina 27695\\
USA\\
\printead{e1}}
\address[B]{Department of Statistics and Probability\\
Michigan State University\\
619 Red Cedar Road\\
C-413 Wells Hall\\
East Lansing, Michigan 48824\\
USA\\
\printead{e2}}
\end{aug}

%
\received{\smonth{11} \syear{2012}}
%
\revised{\smonth{12} \syear{2014}}

%
\begin{abstract}
Let $X = \{X(t), t\in\R^{N} \}$ be a centered Gaussian random field
with stationary increments and
$X(0) = 0$. For any compact rectangle $T \subset\R^N$ and $u\in\R$,
denote by
$A_u = \{t\in T\dvtx X(t)\geq u\}$ the
excursion set. Under $X(\cdot) \in C^2(\R^N)$ and certain regularity
conditions, the mean Euler
characteristic of~$A_u$, denoted by $\E\{\varphi(A_u)\}$, is derived.
By applying the Rice method, it is shown that,
as $u \to\infty$, the excursion probability $\P\{\sup_{t\in T} X(t)
\geq u \}$ can be approximated by
$\E\{\varphi(A_u)\}$ such that the error is exponentially smaller than
$\E\{\varphi(A_u)\}$. This verifies the expected Euler characteristic
heuristic
for a large class of Gaussian random fields with stationary increments.
\end{abstract}

%
\end{frontmatter}

\section{Introduction}\label{sec1}
Let $X = \{X(t), t\in T \}$ be a real-valued Gaussian random field on
probability
space $(\Omega, \mathcal{F}, \P)$, where $T$ is the parameter set. The
study of
the excursion probability $\P\{\sup_{t\in T} X(t) \geq u \}$ is a
classical but very
important problem in probability theory and has many applications in statistics
and related areas. Many authors have developed various methods for precise
approximations of $\P\{\sup_{t\in T} X(t) \geq u \}$. These include the
double sum method [\citet{Piterbarg96book}], the tube method
[\citet{Sun93}], the
Euler characteristic method [\citet{Adler00}, \citet{TTA03},
Taylor, Takemura and Adler (\citeyear{TTA05}), \citet{AT07}] and the Rice method
[Aza{\"{\i}}s, Bardet and
Wschebor (\citeyear{ABW02}), Aza\"is and Delmas (\citeyear{AD02}), Aza\"is and Wschebor
(\citeyear{AW05,AzaisW08,AW09})].

For a centered, unit-variance smooth Gaussian random field $X = \{X(t),
t\in T \}$ parameterized
on a manifold $T$, Adler and Taylor [(\citeyear{AT07}), Theorem~14.3.3] proved,
under certain conditions on
the regularity of $X$ and topology of $T$, the following approximation:
%
\begin{equation}
\label{Eq:MEC approx error} \P \Bigl\{\sup_{t\in T} X(t) \geq u \Bigr\} = \E
\bigl\{\varphi(A_u)\bigr\}\bigl(1 + o \bigl(e^{- \alpha u^2}\bigr)
\bigr) \qquad\mbox{as } u\to\infty,
\end{equation}
where $\varphi(A_u)$ is the Euler characteristic of excursion set $A_u
= \{t\in T\dvtx X(t)\geq u\}$
and $\alpha> 0$ is a constant which relates to the curvature of the boundary
of $T$ and the second-order partial derivatives of $X$. This verifies
the ``Expected Euler
Characteristic Heuristic'' for unit-variance smooth Gaussian random
fields. We refer
to Takemura and Kuriki (\citeyear{TakKurik02}), \citet{TTA03} and Taylor,
Takemura and Adler (\citeyear{TTA05})
for similar results in special cases. It should be mentioned that
Taylor, Takemura and Adler (\citeyear{TTA05})
were able to provide an explicit form of $\alpha$ in~(\ref{Eq:MEC
approx error}).

The approximation (\ref{Eq:MEC approx error}) is remarkable and very
accurate, since
$\E\{\varphi(A_u)\}$ is computable and the error is exponentially
smaller than this principal
term. It has been applied for $P$-value approximation in many
statistical applications
to brain imaging, cosmology and environmental sciences. We refer to
\citet{AT07} and its forthcoming companion Adler, Taylor and
Worsley (\citeyear{AdlerTW12})
for further
information.
However, the above requirement of ``constant variance'' on the Gaussian
random fields
is too restrictive for many applications and excludes some important
Gaussian random fields
such as those with stationary increments (see Section~\ref{sec2} below), or more
generally,
Gaussian random intrinsic functions [\citet{Matheron73}, Stein
(\citeyear{Stein99},
\citeyear{Stein12})]. If the
constant variance condition on $X$ is not satisfied, then several
important properties
[e.g., $X(t)$ and its gradient $\nabla X(t)$ are independent for every
$t$] are not
available and the formulas for computing
$\E\{\varphi(A_u)\}$ [cf. Theorems 12.4.1 and 12.4.2 in \citet{AT07}]
cannot be applied. Little had been known on whether the approximation~(\ref{Eq:MEC approx error}) still holds. The only exception is Aza\"is
and Wschebor
[(\citeyear{AzaisW08}), Theorem~5], where they proved~(\ref{Eq:MEC approx error}) for a
centered smooth
Gaussian random field $X$ whose maximum variance is attained in the
interior of $T$.

In this paper, let $X = \{X(t), t\in\R^N \}$ be a centered
real-valued Gaussian
random field with stationary increments and $X(0) = 0$, and let $T
\subset\R^N$
be a rectangle. Our objectives are to compute the expected Euler characteristic
$\E\{\varphi(A_u)\}$ and to show that it can be applied to give an accurate
approximation for the excursion probability $\P\{\sup_{t\in T} X(t)
\geq u \}$.
In particular, we prove that (\ref{Eq:MEC approx error})
holds for a large class of smooth Gaussian random fields with
stationary increments
and $X(0) = 0$. In this generality, our main results in Sections~\ref{sec3} and
\ref{sec4} are new even
for the case of $N=1$.


The paper is organized as follows. In Section~\ref{sec2}, we provide
some preliminaries on Gaussian random fields with stationary increments and
prove some basic lemmas. These are derived from the spectral representation
of the random fields and will be useful for proving the main results in
Sections~\ref{sec3}
and~\ref{sec4}.

In Section~\ref{sec3}, we compute the mean Euler characteristic $\E\{\varphi
(A_u)\}$ by applying
the Kac--Rice metatheorem in Adler and Taylor [(\citeyear{AT07}),
Theorem~11.2.1]
[see also Adler and
Taylor (\citeyear{AdlerT11}), Theorem~4.1.1]. The computation of $\E\{
\varphi(A_u)\}$
involves the
conditional expectation of the determinant of the Hessian $\nabla^2
X(t)$ given $X(t)$ and $\nabla X (t)$,
which is more complicated for random fields with nonconstant variance function.
For Gaussian random fields with stationary increments, we are able to
make use of
the properties of $\nabla X$ and $\nabla^2 X$ (e.g., their
stationarity) to provide
an explicit formula in Theorem~\ref{Thm:MEC} for $\E\{\varphi(A_u)\}$,
using only derivatives of up to second order of the covariance function.

Section~\ref{sec4} is the core part of this paper. Theorems \ref{Thm:MEC
approximation 1} and
\ref{Thm:MEC approximation 2} provide approximations to the excursion
probability which
are analogous to (\ref{Eq:MEC approx error}) for Gaussian random fields
with stationary
increments and $X(0) = 0$. Since these random fields do not have
constant variance, it
is not clear if the original method for proving Theorem~14.3.3 in
\citet{AT07}
is still applicable. Instead, our argument is based on
the Rice method in Aza\"is and Delmas (\citeyear{AD02}) [see also Adler and Taylor
(\citeyear{AT07}), pages 96--99].
More specifically, we decompose the rectangle $T$ into several faces of
lower dimensions
and then apply the idea of \citet{Piterbarg96} and the Bonferroni
inequality to derive
upper and lower bounds for $\P\{\sup_{t\in T} X(t) \geq u \}$ in terms
of the number of
extended outward maxima [see (\ref{Ineq:upperbound}), (\ref{Ineq:lowerbound})]
and local maxima [see (\ref{Ineq:lowerbound and upperbound})], respectively.
The main idea is to show that, in both cases, the upper bound makes the
major contribution
for estimating $\P\{\sup_{t\in T} X(t) \geq u \}$ and the last two
terms in the lower bounds
in (\ref{Ineq:lowerbound}) and (\ref{Ineq:lowerbound and upperbound})
are super-exponentially small. Under a mild technical condition on the
variogram of $X$,
we apply (\ref{Ineq:lowerbound and upperbound}) to obtain in Theorem~\ref{Thm:MEC approximation 1} an expansion of the excursion probability
which is, in spirit, similar to the case of stationary Gaussian fields
[cf. (14.0.3) in
\citet{AT07}]. Theorem~\ref{Thm:MEC approximation 2}
establishes a general
approximation to $\P\{\sup_{t\in T} X(t) \geq u \}$ in terms of $\E\{
\varphi(A_u)\}$, which
verifies the ``Expected Euler Characteristic Heuristic'' for smooth
Gaussian random fields with
stationary increments and $X(0) = 0$. For the purpose of comparison, we
mention that, if
$Z = \{Z(t), t \in\R^N\}$ is a real-valued, centered stationary
Gaussian random field, then
the random field $X$ defined by $X(t) = Z(t) - Z(0)$ has stationary
increments with $X(0) = 0$.
Consequently, Theorems \ref{Thm:MEC approximation 1} and
\ref{Thm:MEC approximation 2} provide approximations to the excursion
probability
$\P\{\sup_{t\in T} Z(t) - Z(0)\geq u \}$.

Section~\ref{sec5} provides further remarks on the main results and some
examples where
significant simplifications can be made. In Remarks \ref{Example:unique
maximum 1}
and \ref{Example:unique maximum 2}, we show that if the variance function
of the random field attains its maximum at a unique point, then one can
apply the
Laplace method to derive a first-order approximation for the excursion
probability explicitly.
Finally, the \hyperref[app]{Appendix} contains proofs of some
auxiliary lemmas.

\section{Gaussian fields with stationary increments}\label{sec2}
\subsection{Spectral representation}\label{sec2.1}
Let $X = \{X(t), t\in\R^{N} \}$ be a real-valued
centered Gaussian random field with stationary increments. That is, for
any $h \in\R^N$,
$\{X(t+h)- X(h), t\in\R^{N} \} \stackrel{d}{=} \{X(t)-X(0), t\in\R
^{N} \}$, where $\stackrel{d}{=}$
means equality in finite dimensional distributions.
We assume that $X$ has continuous covariance function $C(t, s) = \E\{
X(t)X(s)\}$. 
Then it is known [cf. \citet{Yaglom57}] that
%
\begin{eqnarray}
\label{Eq:Yaglom} %
&&C(t, s) - C(t, 0) - C(0, s) + C(0,0)
\nonumber
\\[-8pt]
\\[-8pt]
\nonumber
&&\qquad= \int_{\R^N} \bigl(e^{i \langle t, \lambda\rangle } - 1\bigr) \bigl(
e^{- i \langle s,
\lambda
\rangle } - 1\bigr) F(d \lambda ) + \langle t, \Theta s\rangle ,
%
\end{eqnarray}
where $\langle x, y\rangle $ is the ordinary inner product in $\R^N$,
$\Theta$ is
an $N\times N$ nonnegative definite matrix and $F $ is a
nonnegative symmetric measure on $\R^N \setminus\{0 \}$ which
satisfies
%
\begin{equation}
\label{Eq:LM} \int_{\R^N} \frac{\|\lambda \|^2} {1 + \|\lambda \|^2} F(d \lambda ) <
\infty.
\end{equation}
Similar to stationary random fields, the measure $F$ and its density
(if it exists) $f(\lambda )$ are called the \textit{spectral measure}
and \textit{spectral density}
of $X$, respectively. It is known that many probabilistic, analytic and
geometric properties of
$\{X(t), t\in\R^{N} \}$ can be described in terms of its
spectral measure $F$ and, on the other hand, various interesting Gaussian
random fields can be constructed by choosing their spectral measures
appropriately.
See \citet{X09}, \citet{XX09} and the references therein for
more information.

By (\ref{Eq:Yaglom}), we see that $X$ has the following stochastic
integral representation:
%
\begin{equation}
\label{Eq:rep} X(t) - X(0) \stackrel{d} {=} \int_{\R^N}
\bigl(e^{i \langle t, \lambda\rangle }-1\bigr) W(d\lambda) + \langle{\mathbf Y}, t\rangle ,
\end{equation}
where ${\mathbf Y}$ is an $N$-dimensional Gaussian random vector and
$W$ is
a complex-valued Gaussian
random measure (independent of ${\mathbf Y}$) with $F$ as its control measure.
Note that in (\ref{Eq:rep}) there is no restriction on $X(0)$ other
than that all joint distributions
of $\{X(t), t \in\R^N\}$ are Gaussian.

For simplicity, we assume throughout this paper that ${\mathbf Y}= 0$. It
follows from (\ref{Eq:Yaglom})
or (\ref{Eq:rep}) that the \emph{variogram} $\nu$ of $X$ is given by
%
\begin{equation}
\label{Eq:variogram} \nu(h) := \E\bigl(X(t+h)-X(t)\bigr)^2 = 2\int
_{\R^N}\bigl(1-\cos\langle h, \lambda \rangle \bigr) F(d\lambda).
\end{equation}

Mean-square directional derivatives and sample path differentiability
of Gaussian random
fields have been well studied. See, for example, \citet{Adler81},
\citet{AT07},
\citet{Potthoff10}, \citet{XX09}. In particular, general sufficient
conditions for a Gaussian
random field to have a modification whose sample
functions are in $C^{k}(\R^N)$ are given by \citet{AT07}.
For a Gaussian random field $X= \{X(t), t \in\R^N\}$ with stationary
increments,
\citet{XX09} provided conditions for its sample path differentiability
in terms of the spectral density function $f(\lambda)$. A similar
argument can be applied to
give the following spectral condition for the sample functions of $X$
to be in $C^{k}(\R^N)$,
whose proof is given in \citet{ChengPhD} and is omitted here.

\begin{proposition}\label{prop:de}
Let $X = \{X(t), t\in\R^{N} \}$ be a real-valued centered Gaussian
random field with
stationary increments and let $k_i\ (1 \le i \le N)$ be nonnegative
integers. If there is a
constant $\ep>0$ such that
%
\begin{equation}
\label{Eq:spCon} \int_{\{\|\lambda\|\ge1\}} \prod_{i=1}^N|
\lambda_i|^{2k_i + \ep} F(d\lambda) < \infty,
\end{equation}
then $X$ has a modification $\widetilde{X}$ such that the partial derivative
$\frac{\partial^k\widetilde{X}(t)}{\partial t_1^{k_1}\cdots\partial
t_N^{k_N}}$ is
continuous on $\R^N$ almost surely, where $k = \sum_{i=1}^N k_i$.
Moreover, for any
compact rectangle $T \subset\R^N$ and any
$\ep' \in(0, \ep\wedge1)$, there exists a constant $c_1$ such that
%
\begin{equation}
\E \biggl(\frac{\partial^k\widetilde{X}(t)}{\partial t_1^{k_1}\cdots
\partial t_N^{k_N}}- \frac{\partial^k\widetilde{X}(s)}{\partial s_1^{k_1}\cdots\partial
s_N^{k_N}} \biggr)^2 \leq
c_1\|t-s\|^{\ep'}\qquad \forall t,s\in T.
\end{equation}
\end{proposition}

For notational simplicity, we will not distinguish $X$ from its
modification $\widetilde{X}$.
As a consequence of Proposition~\ref{prop:de}, we see that, if $X= \{
X(t), t\in\R^{N} \}$
has a spectral density $f(\lambda)$ which satisfies
%
\begin{equation}
\label{Eq:spCon2} f(\lambda)= O \biggl(\frac{1}{\|\lambda \|^{N+2k+H}} \biggr)\qquad
 \mbox{as } \|\lambda
\| \to\infty,
\end{equation}
for some integer $k \geq1$ and $H\in(0, 1)$, then the sample
functions of $X$ are in $C^k(\R^N)$
a.s. Further examples of anisotropic Gaussian random fields which may
have different
smoothness along different directions can be found in \citet{XX09}.

When $X(\cdot) \in C^2(\R^N)$ almost surely, we write $\frac
{\partial
X(t)}{\partial t_i} = X_i (t)$ and
$\frac{\partial^2 X(t)}{\partial t_i\,\partial t_j} = X_{ij}(t)$. We will
use the same notation for
the partial derivatives of deterministic functions such as $\nu( \cdot
)$ in Theorem~\ref{Thm:MEC approximation 1}.

Denote by $\nabla X(t)$ and $\nabla^2 X(t)$ the column vector $(X_1(t),
\ldots, X_N(t))^{T}$ and the
$N\times N$ matrix $(X_{ij}(t))_{ i, j = 1, \ldots, N}$, respectively.
It follows from (\ref{Eq:Yaglom})
that for every $t \in\R^N$,
%
\begin{equation}
\label{Eq:laij} \lambda _{ij}:= \int_{\R^N} \lambda
_i \lambda _j F(d \lambda ) = \frac{\partial^2 C(t, s)}{
\partial t_i\, \partial s_j}
\bigg|_{s=t}= \E\bigl\{X_i(t) X_j(t)\bigr\}.
\end{equation}
Let $\La= (\lambda _{ij})_{i, j = 1,\ldots, N}$, then (\ref
{Eq:laij}) shows
that $\La=\operatorname{ Cov}(\nabla X(t))$
for all $t$. In particular, the distribution of $\nabla X(t)$ is
independent of $t$. Let
\[
\lambda _{ij}(t) := \int_{\R^N} \lambda
_i \lambda _j \cos\langle t, \lambda \rangle F(d \lambda ),\qquad
\La(t):= \bigl(\lambda _{ij}(t)\bigr)_{i, j = 1,\ldots, N}. %
\]
Then we have
%
\begin{eqnarray}
\label{ladiff} %
\lambda _{ij}(t)-\lambda _{ij} &=&
\int_{\R^N} \lambda _i \lambda _j \bigl(\cos
\langle t, \lambda \rangle -1 \bigr) F(d\lambda )
\nonumber
\\[-8pt]
\\[-8pt]
\nonumber
&=& \frac{\partial^2 C(t, s)}{\partial s_i\, \partial s_j}\bigg|_{s=t}=\E \bigl\{\bigl(X(t)-X(0)
\bigr)X_{ij}(t)\bigr\}, %
\end{eqnarray}
or equivalently, $\La(t)-\La= \E\{(X(t)-X(0)) \nabla^2 X(t)\}$.

In studies of a Gaussian random field $X$ with stationary increments in
the literature, it is
often assumed that $X(0) = 0$. In this case, $\La(t)-\La= \E\{X(t)
\nabla^2 X(t)\}$. With
little loss of generality, we will follow this convention by assuming
$X(0) = 0$ in the rest
of this paper. The general case can be dealt with by first applying the
result of this paper to
the random field $\{X(t) - X(0), t \in\R^N\}$ and then taking into
consideration of the
available information on $X(0)$ [an interesting special case is when
$X(0)$ is independent
of $\{X(t) - X(0), t \in\R^N\}$].

\subsection{Hypotheses and some important properties}\label{sec2.2}


Let $T=\prod^N_{i=1}[a_i, b_i]$ be a compact rectangle in $\R^N$, where
$a_i<b_i$ for all $1\leq i \leq N$
and $0\notin T$ (the case of $0\in T$ will be discussed in Remark~\ref
{Remark:contain the origin}).
In addition to assuming that the Gaussian random field $X= \{X(t), t
\in\R^N\}$ has stationary increments
and $X(0) = 0$,
we will make use of the following conditions:
\begin{longlist}[(H3$'$)]
\item[(H1)] $X(\cdot) \in C^2(T)$ almost surely and its second
derivatives satisfy the
\emph{uniform mean-square H\"older condition}: there exist constants
$L>0$ and $ \eta\in(0, 1]$
such that
%
\begin{equation}
\E\bigl(X_{ij}(t)-X_{ij}(s)\bigr)^2 \leq L
\|t-s\|^{2\eta}\qquad \forall t,s\in T, i, j= 1, \ldots, N.
\end{equation}
\item[(H2)] For every $t\in T$, the matrix $\La-\La(t)$
is nondegenerate.
\item[(H3)] For every pair $(t, s)\in T^2$ with $t\neq s$, the
Gaussian random vector
\[
\bigl(X(t), \nabla X(t), X_{ij}(t), X(s), \nabla X(s),
X_{ij}(s), 1\leq i\leq j\leq N\bigr)
\]
is nondegenerate.
\item[(H3$'$)] For every $t\in T$, $(X(t), \nabla X(t),
X_{ij}(t), 1\leq i\leq j\leq N)$
is nondegenerate.
\end{longlist}

Clearly, by Proposition~\ref{prop:de}, condition (H1) is
satisfied if (\ref{Eq:spCon2})
holds for $k=2$. Also note that (H3) implies (H3$'$). We
shall use conditions (H1),
(H2) and (H3) to prove Theorems \ref{Thm:MEC approximation
1} and \ref{Thm:MEC approximation 2}
on the excursion probability. Condition (H$3'$) will be used for
computing $\E\{\varphi(A_u)\}$
in Theorem~\ref{Thm:MEC}.

We point out that the nondegeneracy conditions (H3) and (H3$'$) are standard for studying
crossing problems when $N=1$, excursion sets and excursion
probabilities of smooth Gaussian random fields.
In the case where $N=1$ and $X$ is a stationary Gaussian process, Cram{\'e}r and Leadbetter
[(\citeyear{CramerL67}), pages~203--204]
showed that (H3$'$) is automatically satisfied if $X$ has
second-order mean square derivatives
and the spectral measure of $X$ is not purely discrete. See Exercises
3.4 and 3.5 in Aza\"is and Wschebor
[(\citeyear{AW09}), page~87] for similar results. Notice that (H3) and
(H3$'$) are equivalent to saying that
the corresponding covariance matrices are nondegenerate which, in turn,
can be verified by establishing
positive lower bounds for the conditional variances. Thus, (H3)
and (H3$'$) are related to
the properties of local nondeterminism [cf. \citet{Cuzick77},
\citet{X09}]. Hence, for a general Gaussian
random field $X$ with stationary increments, it is possible to provide
sufficient conditions in terms of
the spectral measure $F$ for (H3) and (H3$'$) to
hold. In
order not to make this paper too
lengthy, we do not give details here.

The following lemma shows that for Gaussian fields with stationary
increments and $X(0) = 0$, (H2)
is equivalent to $\La-\La(t)$ being positive definite.

\begin{lemma} \label{Lem:posdef} For every $t\in\R^N$, $\La-\La
(t)$ is
nonnegative definite.
Hence, under \textup{(H2)}, $\La-\La(t)$ is positive definite for every
$t \in T$.
\end{lemma}

\begin{pf} Let $t \in\R^N $ be fixed. It follows from (\ref
{ladiff}) that for any $(a_1, \ldots, a_N) \in\R^N\setminus\{0\}$,
%
\begin{equation}
\label{Eq:nonnegative def} %
\sum_{i,j=1}^N
a_ia_j \bigl(\lambda _{ij}-\lambda
_{ij}(t)\bigr) = \int_{\R
^N} \Biggl(\sum
_{i=1}^N a_i\lambda _i
\Biggr)^2 \bigl(1- \cos\langle t, \lambda \rangle \bigr) F(d\lambda ). %
\end{equation}
Since $(\sum_{i=1}^N a_i\lambda _i)^2 (1- \cos\langle t, \lambda
\rangle ) \geq0$ for all
$\lambda \in\R^N$,
(\ref{Eq:nonnegative def}) is always nonnegative, which implies that
$\La-\La(t)$ is nonnegative definite.
If (H2) is satisfied, then, for every $t \in T$, all the
eigenvalues of $\La-\La(t)$ are positive.
This completes the proof.
\end{pf}

It follows from (\ref{Eq:nonnegative def}) that, if the spectral
measure $F$ is \textit{full} [i.e.,
not supported on any $(N- 1)$-dimensional hyperplane], then (H2)
holds. Hence, (H2)
is in fact a mild condition for smooth Gaussian fields with stationary
increments.


Lemma~\ref{Lem:posdef} and the following two lemmas indicate some significant
properties of Gaussian fields with stationary increments.
They will play important roles in later sections.

\begin{lemma}\label{Lem:independent}
Let $t \in\R^N$ be fixed. Then for all $i,j,k$, the random variables
$X_i(t)$ and $X_{jk}(t)$
are independent. Moreover, $\E\{X_{ij}(t)X_{kl}(t)\}$ is symmetric in
$i$, $j$, $k$, $l$.
\end{lemma}

\begin{pf} By (\ref{Eq:Yaglom}), one can verify that for $t, s
\in\R^N$,
\[
\E\bigl\{X_i(t)X_{jk}(s)\bigr\} = \frac{\partial^{3} C(t, s)}{\partial t_i
\,\partial s_j\, \partial s_k} =
\int_{\R^N} \lambda _i\lambda _j\lambda
_k \sin\langle t-s, \lambda \rangle F(d \lambda ).
\]
Letting $s=t$ we see that $X_i(t)$ and $X_{jk}(t)$ are independent.
Similarly, we have
\[
\E\bigl\{X_{ij}(t)X_{kl}(s)\bigr\} = \frac{\partial^{4} C(t, s)}{
\partial t_i\, \partial t_j\, \partial s_k\, \partial s_l}=
\int_{\R
^N} \lambda _i\lambda _j
\lambda _k\lambda _l \cos\langle t-s, \lambda \rangle F(d
\lambda ).
\]
This implies the second conclusion.
\end{pf}

The following lemma is a consequence of Lemma~\ref{Lem:independent}.

\begin{lemma}\label{Lem:symmetric}
Let $A = (a_{ij})_{1\leq i, j \leq N}$ be a symmetric matrix. Then for
any fixed $t \in\R^N$,
\[
\mathcal{S}_t(i,j,k,l)= \E\bigl\{\bigl(A \nabla^2 X(t)
A\bigr)_{ij} \bigl(A \nabla^2 X(t) A\bigr)_{kl}
\bigr\}
\]
is a symmetric function of $i$, $j$, $k$, $l$.
\end{lemma}

\section{The mean Euler characteristic}\label{sec3}
\subsection{Related existing results and notation}\label{sec3.1}
The rectangle $T=\prod^N_{i=1}[a_i, b_i]$ can be decomposed into several
faces of lower dimensions. We use the same notation as in Adler and
Taylor [(\citeyear{AT07}), page~134].

A face $J$ of dimension $k$, is defined by fixing a subset $\sigma(J)
\subset\{1, \ldots, N\}$ of size $k$ [if $k=0$, we have $\sigma(J)
= \varnothing$ by convention] and a subset $\varepsilon(J) = \{
\varepsilon_j,
j\notin\sigma(J)\} \subset\{0, 1\}^{N-k}$ of size $N-k$, so that
\begin{eqnarray*}
J&=& \bigl\{ t=(t_1, \ldots, t_N) \in T
\dvtx a_j< t_j <b_j \mbox{ if} j\in
\sigma(J),
\nonumber
\\[-8pt]
\\[-8pt]
\nonumber
&&\hspace*{56pt}t_j = (1-\ep_j)a_j +
\ep_{j}b_{j} \mbox{ if} j\notin\sigma(J) \bigr\}.
\end{eqnarray*}
Denote by $\partial_k T$ the collection of all $k$-dimensional faces in
$T$, then
the interior of $T$ is given by $\stackrel{\circ}{T}=\partial_N T$ and
the boundary
of $T$ is given by $\partial T = \bigcup^{N-1}_{k=0}\bigcup_{J\in
\partial_k
T} J$. For
$J\in\partial_k T$, denote by
$\nabla X_{|J}(t)$ and $\nabla^2 X_{|J}(t)$ the column vector
$(X_{i_1} (t),
\ldots, X_{i_k} (t))^T_{i_1, \ldots, i_k \in\sigma(J)}$ and the
$k\times k$ matrix
$(X_{mn}(t))_{m, n \in\sigma(J)}$, respectively.

If $X(\cdot)\in C^2(\R^N)$ and it is a Morse function a.s. [cf.
Definition~9.3.1 in
\citet{AT07}], then according to Corollary~9.3.5 or pages
211--212 in \citet{AT07}, the Euler characteristic of the
excursion set
$A_u = \{t\in T\dvtx X(t)\geq u\}$ is given by
%
\begin{equation}
\label{Eq:def of Euler charac} \varphi(A_u)= \sum^N_{k=0}
\sum_{J\in\partial_k T}(-1)^k\sum
^k_{i=0} (-1)^i \mu_i(J)
\end{equation}
with
%
\begin{eqnarray}
\label{Eq:mu_iJ} %
\mu_i(J) & :=& \# \bigl\{ t\in J\dvtx X(t)
\geq u, \nabla X_{|J}(t)=0, \operatorname {index} \bigl(\nabla^2
X_{|J}(t)\bigr)=i,
\nonumber
\\[-8pt]
\\[-8pt]
\nonumber
&& \hspace*{113pt}\varepsilon^*_jX_j(t) \geq0 \mbox{ for all } j\notin
\sigma(J) \bigr\}, %
\end{eqnarray}
where $\ep^*_j=2\ep_j-1$ and the index of a matrix is defined as the
number of its negative
eigenvalues. We also define
%
\begin{equation}
\label{Eq:mu_iJ2} \wt{\mu}_i(J) := \# \bigl\{ t\in J\dvtx X(t)\geq u,
\nabla X_{|J}(t)=0, \operatorname {index} \bigl(\nabla^2
X_{|J}(t)\bigr)=i \bigr\}.
\end{equation}

It follows from (\ref{Eq:variogram}) that $\nu(t) = \operatorname{Var} (X(t))$.
Let $\sigma^2_T =
\sup_{t\in T} \nu(t) $ be the maximum variance. For any $t \in T$ and
$J\in\partial_k T$,
where $k \ge1$, let
%
\begin{eqnarray}
\label{Def:EJ} %
\La_J& =&(\lambda _{ij})_{i, j \in\sigma(J)}
= \operatorname{ Cov}\bigl(\nabla X_{|J}(t)\bigr),\nonumber\\
 \La_J(t)& =&
\bigl(\lambda _{ij}(t)\bigr)_{i, j \in\sigma(J)},\nonumber
\\
\theta_{J,t}^2 &=& \operatorname{Var} \bigl(X(t)| \nabla
X_{|J}(t)\bigr), \qquad\gamma_t^2 = \operatorname{Var}
\bigl(X(t)| \nabla X(t)\bigr),
\\
\{J_1, \ldots, J_{N-k}\}&= &\{1, \ldots, N\}\setminus
\sigma(J),\nonumber
\\
E(J) &=& \bigl\{(t_{J_1}, \ldots, t_{J_{N-k}})\in
\R^{N-k}\dvtx t_j\ep _j^*\geq0, j=
J_1, \ldots, J_{N-k}\bigr\}.\nonumber %
\end{eqnarray}
%
Note that $\theta_{J,t}^2 \geq\gamma_t^2$ for all $t\in T$ and
$\theta
_{J,t}^2= \gamma_t^2$ if $J = \partial_N T$.
If $J = \{\tau\}\in\partial_0 T$ is a vertex, then $\nabla X_{|J}(t)$
is not defined and we set $\theta_{J, t}^2$ as
$\nu(t)$ by convention. Moreover, if $J= \{\tau\}\in\partial_0 T$,
then $E(\{\tau\})$ is a quadrant of $\R^N$ decided by the corresponding
$\varepsilon(\{\tau\}) \in\{0, 1\}^N$. In the sequel, we will write
$\theta_{J,t}^2$ as $\theta_t^2$ for simplicity of notation.
This will not cause any confusion because $\theta_{t}^2$ always appears
together with $t \in J$.


For $t \in T$, let $C_{j}(t)$ be the $(1, j+1)$ entry of $(\operatorname{Cov}
(X(t), \nabla X(t)))^{-1}$, that is,
\[
C_{j}(t)= M_{1,j+1}(t)/\operatorname{ det Cov} \bigl(X(t), \nabla X(t)
\bigr),
\]
where $M_{1,j+1}(t)$ is the cofactor of the
$(1, j+1)$ entry, $\E\{X(t)X_j(t) \}$, in the covariance matrix $\operatorname
{Cov} (X(t), \nabla X(t))$. If $\{X(t), t \in\R^N\}$ is replaced by
a Gaussian field $\{Z(t), t \in\R^N\}$ with constant variance, the
independence of $Z(t)$ and $\nabla Z(t)$ for each $t$
implies that $M_{1,j+1}(t)$ and hence $C_{j}(t)$ is zero for all $j\ge1$.

Denote by $H_k(x)$ the Hermite polynomial of order $k$, that is,
$H_k(x) = (-1)^k e ^{x^2/2} \frac{d^k}{dx^k}( e^{-x^2/2} )$.\vspace*{1pt}
Then it can be verified directly [cf. Adler and Taylor (\citeyear
{AT07}), page~289] that
%
\begin{equation}
\label{Eq:Hermite} \int_u^\infty H_k(x)
e^{-x^2/2} \,dx = H_{k-1}(u) e^{-u^2/2},
\end{equation}
where $u>0$ and $k\geq1$. For a matrix $A$, let $|A|$ denote its determinant.
Let $\R_+ = [0,\infty)$, $\R_- = (-\infty, 0]$ and let $\Psi(u) =
(2\pi
)^{-1/2}\int_u^\infty e^{-x^2/2} \,dx$.

\subsection{Computing the mean Euler characteristic}\label{sec3.2}

The following lemma is an extension of Lemma~11.7.1 in \citet
{AT07}. It
provides a key step for computing the mean Euler characteristic in
Theorem~\ref{Thm:MEC} below,
and has a close connection with Theorem~\ref{Thm:MEC approximation 1}.
It follows from (\ref{Eq:expectation of sum}) that $(-1)^k\E\{\sum^k_{i=0} (-1)^i \wt{\mu}_i (J)\}$
is always positive. This fact will be used to approximate the expected
number of local maxima
above level $u$; see Lemma~\ref{Lem:mean M_u}.

\begin{lemma} \label{Lem:MEC}
Let $X = \{X(t), t\in\R^{N} \}$ be a centered Gaussian random field
with stationary
increments and $X(0) = 0$. Suppose conditions \textup{(H1)}, \textup{(H2)} and \textup{(H3$'$)} hold. Then for each $J
\in\partial_k T$ with $k \geq1$,
%
\begin{eqnarray}
\label{Eq:expectation of sum} %
&&\E \Biggl\{\sum^k_{i=0}
(-1)^i \wt{\mu}_i (J) \Biggr\}
\nonumber
\\[-8pt]
\\[-8pt]
\nonumber
&&\qquad= \frac
{(-1)^k}{(2\pi
)^{(k+1)/2}|\La_J|^{1/2}} \int
_J \frac{|\La_J-\La_J(t)|}{\theta^k_t} H_{k-1} \biggl(
\frac
{u}{\theta
_t} \biggr)e^{-u^2/(2\theta^2_t)} \,dt. %
\end{eqnarray}
\end{lemma}

\begin{pf}Let $\mathcal{D}_i$ be the collection of all $k \times k
$ matrices with index $i$.
Recall the definition of $\wt{\mu}_i (J)$ in (\ref{Eq:mu_iJ2}), and
thanks to (H1) and
(H3$'$), we can apply the Kac--Rice metatheorem [cf. Theorem~11.2.1 or Corollary~11.2.2 in \citet{AT07}] to get that the
left-hand side of (\ref
{Eq:expectation of sum}) becomes
%
\begin{eqnarray}
\label{Eq:mean mu Kac-Rice} %
&&\int_{J} p_{\nabla X_{|J}(t)}(0) \,dt
\nonumber
\\[-8pt]
\\[-8pt]
\nonumber
&&\qquad{}\times
\sum^k_{i=0} (-1)^i \E\bigl\{
\bigl|\operatorname {det} \nabla^2 X_{|J}(t)\bigr| \mathbh{1}_{\{\nabla^2 X_{|J}(t)
\in\mathcal{D}_i\}}
\mathbh{1}_{\{
X(t)\geq u\}} | \nabla X_{|J}(t)=0 \bigr\}.\hspace*{-20pt} %
\end{eqnarray}
Note that on the set $\mathcal{D}_i$, the matrix $\nabla^2 X_{|J}(t)$
has $i$ negative eigenvalues,
which implies $(-1)^i|\operatorname{det} \nabla^2 X_{|J}(t)|= \operatorname{ det}
\nabla
^2 X_{|J}(t)$. Also,
$\bigcup_{i=0}^k \{\nabla^2 X_{|J}(t) \in\mathcal{D}_i\}=\Omega$ a.s.,
hence (\ref{Eq:mean mu Kac-Rice})
equals
%
\begin{eqnarray}
\label{Eq:mean mu} %
&&\int_{J} p_{\nabla X_{|J}(t)}(0) \,dt\,
\E\bigl\{ \operatorname{det} \nabla^2 X_{|J}(t)
\mathbh{1}_{\{X(t)\geq u\}} | \nabla X_{|J}(t)=0 \bigr\}
\nonumber
\\
&&\qquad = \int_{J} \,dt\int^\infty_u
\,dx \frac{e^{-x^2/(2\theta
^2_t)}}{(2\pi)^{(k+1)/2}|\La_J|^{1/2}\theta_t}\\
&&\quad\qquad{}\times \E\bigl\{\operatorname{det} \nabla^2
X_{|J}(t) | X(t)=x, \nabla X_{|J}(t)=0 \bigr\}.\nonumber %
\end{eqnarray}

Now we turn to computing $\E\{\operatorname{det} \nabla^2 X_{|J}(t) | X(t)=x,
\nabla X_{|J}(t)=0 \}$ in (\ref{Eq:mean mu}).
By Lemma~\ref{Lem:posdef}, under (H2), $\La-\La(t)$, and hence
$\La_J-\La_J(t)$ are
positive definite for every $t\in J$. Thus, there
exists a $k \times k $ positive definite matrix $Q_t$ such that
%
\begin{equation}
\label{Eq:decomposition of positivedef matrix} Q_t\bigl(\La_J-\La_J(t)
\bigr)Q_t = I_k,
\end{equation}
where $I_k$ is the $k\times k$ identity matrix. It follows from (\ref
{ladiff}) that
$\La_J(t)-\La_J=\E\{X(t)\nabla^2 X_{|J}(t)\}$. Hence,
\[
\E\bigl\{X(t) \bigl(Q_t \nabla^2 X_{|J}(t)
Q_t\bigr)_{ij}\bigr\} = -\bigl(Q_t\bigl(
\La_J-\La _J(t)\bigr)Q_t
\bigr)_{ij} = -\delta_{ij},
\]
where $\delta_{ij}$ is the Kronecker delta function. We write
%
\begin{equation}
\label{Eq:mean of det} \E\bigl\{\operatorname{det} \bigl(Q_t \nabla^2
X_{|J}(t) Q_t\bigr) | X(t)=x, \nabla X_{|J}(t)=0
\bigr\} = \E\bigl\{\operatorname{det} \Delta(t, x)\bigr\},\hspace*{-20pt}
\end{equation}
where $\Delta(t, x) = (\Delta_{ij}(t, x))_{i,j\in\sigma(J)}$ with all
elements $\Delta_{ij}(t, x)$
being Gaussian variables. To study $\Delta(t, x)$, we only need to find
the mean and covariance
of $\Delta_{ij}(t, x)$.
Note that $\nabla X(t)$ and $\nabla^2 X(t)$ are independent by Lemma~\ref{Lem:independent},
thus
%
\begin{eqnarray}
\label{Eq:expectation of Delta} %
\E\bigl\{ \Delta_{ij} (t, x) \bigr\} &=& \E
\bigl\{\bigl(Q_t \nabla^2 X_{|J}(t)
Q_t\bigr)_{ij} | X(t)=x, \nabla X_{|J}(t)=0
\bigr\}\nonumber
\\
&=&\bigl(\E\bigl\{X(t) \bigl(Q_t \nabla^2
X_{|J}(t) Q_t\bigr)_{ij}\bigr\}, 0, \ldots, 0
\bigr)\nonumber\\
&&{}\times \bigl( \operatorname {Cov} \bigl(X(t), \nabla X_{|J}(t)\bigr)
\bigr)^{-1} (x, 0, \ldots, 0)^T
\\
&=& (-\delta_{ij}, 0, \ldots, 0) \bigl( \operatorname{Cov} \bigl(X(t), \nabla
X_{|J}(t)\bigr) \bigr)^{-1} (x, 0, \ldots, 0)^T\nonumber
\\
&=& -\frac{x}{\theta^2_t}\delta_{ij},\nonumber %
\end{eqnarray}
where the last equality comes from the fact that the $(1,1)$ entry of
$( \operatorname{Cov} (X(t),\break
\nabla X_{|J}(t)) )^{-1}$ is $\operatorname{ det Cov} (\nabla X_{|J}(t))/\operatorname{ det
Cov} (X(t), \nabla X_{|J}(t))=1/\theta^2_t$.
For the covariance, we have
\begin{eqnarray*}
&&\E\bigl\{ \bigl(\Delta_{ij} (t, x)-\E\bigl\{
\Delta_{ij} (t, x)\bigr\}\bigr) \bigl(\Delta_{kl}(t, x)-\E
\bigl\{\Delta_{kl} (t, x)\bigr\}\bigr)\bigr\}
\\
&&\qquad= \E\bigl\{\bigl(Q_t \nabla^2 X_{|J}(t)
Q_t\bigr)_{ij}\bigl(Q_t \nabla^2
X_{|J}(t) Q_t\bigr)_{kl}\bigr\}\\
&&\qquad\quad{}-\bigl(\E\bigl
\{X(t) \bigl(Q_t \nabla^2 X_{|J}(t)
Q_t\bigr)_{ij}\bigr\}, 0, \ldots, 0\bigr)
\\
&&\qquad\quad{} \times\bigl( \operatorname{Cov} \bigl(X(t), \nabla X_{|J}(t)\bigr)
\bigr)^{-1} \bigl(\E\bigl\{ X(t) \bigl(Q_t
\nabla^2 X_{|J}(t) Q_t\bigr)_{kl}
\bigr\}, 0, \ldots, 0\bigr)^T
\\
&&\qquad= \mathcal{S}_t(i,j,k,l)-(-\delta_{ij}, 0, \ldots, 0)
\bigl( \operatorname{Cov} \bigl(X(t), \nabla X_{|J}(t)\bigr)
\bigr)^{-1} (-\delta_{kl}, 0, \ldots, 0)^T
\\
&&\qquad= \mathcal{S}_t(i,j,k,l)-\frac{\delta_{ij}\delta_{kl}}{\theta^2_t}, %
\end{eqnarray*}
where $\mathcal{S}_t$ is a symmetric function of $i$, $j$, $k$, $l$ by
Lemma~\ref{Lem:symmetric} with $A$ replaced by~$Q_t$.
Therefore, (\ref{Eq:mean of det}) becomes
\begin{eqnarray*}
&&\E \biggl\{\frac{1}{\theta_t^k}\operatorname{det} \bigl(\theta_t
Q_t\bigl(\nabla^2 X_{|J}(t)
\bigr)Q_t\bigr) \bigg| X(t)=x, \nabla X_{|J}(t)=0 \biggr\}\\
&&\qquad =
\frac{1}{\theta_t^k} \E \biggl\{\operatorname{det} \biggl(\wt{\Delta }(t)-
\frac
{x}{\theta_t}I_k \biggr) \biggr\}, %
\end{eqnarray*}
where $\wt{\Delta}(t)=(\wt{\Delta}_{ij}(t))_{i,j\in\sigma(J)}$
and all
entries $\wt{\Delta}_{ij}(t)$ are Gaussian variables satisfying
\[
\E\bigl\{\wt{\Delta}_{ij}(t)\bigr\}=0,\qquad \E\bigl\{\wt{
\Delta}_{ij}(t) \wt{\Delta }_{kl}(t)\bigr\}=
\theta_t^2\mathcal{S}_t(i,j,k,l)-
\delta_{ij}\delta_{kl}.
\]
By Corollary~11.6.3 in \citet{AT07}, (\ref{Eq:mean of det})
is equal to $(-1)^k\theta_t^{-k}\times  H_k(x/\theta_t)$, hence
\begin{eqnarray*}
&&\E\bigl\{\operatorname{det} \nabla^2 X_{|J}(t) |
X(t)=x, \nabla X_{|J}(t)=0 \bigr\}
\\
&&\qquad= \E\bigl\{\operatorname{det} \bigl(Q^{-1}_tQ_t
\nabla^2 X_{|J}(t)Q_tQ^{-1}_t
\bigr) | X(t)=x, \nabla X_{|J}(t)=0 \bigr\}
\\
&&\qquad= \bigl|\La_J-\La_J(t)\bigr| \E\bigl\{\operatorname{det}
\bigl(Q_t\nabla^2 X_{|J}(t)Q_t
\bigr) | X(t)=x, \nabla X_{|J}(t)=0 \bigr\}
\\
&&\qquad= \frac{(-1)^k}{\theta_t^k} \bigl|\La_J-\La_J(t)\bigr|H_k
\biggl(\frac
{x}{\theta
_t} \biggr). %
\end{eqnarray*}
Plugging this into (\ref{Eq:mean mu}) and applying (\ref{Eq:Hermite}),
we obtain the desired result.
\end{pf}

The following is the main theorem of this section, which is an
extension of Theorem~11.7.2 of
\citet{AT07} to Gaussian random fields with stationary
increments. Notice that in
\eqref{Eq:MEC}, for every $\{t\}\in\partial_0 T$, $\nabla X(t) \in
E(\{
t\})$ specifies the signs of
the partial derivatives $X_j(t)$ ($j= 1, \ldots, N$) and, for $J \in
\partial_k T$, the set
$\{J_1, \ldots, J_{N-k}\}$ is defined in \eqref{Def:EJ}.

\begin{theorem} \label{Thm:MEC}
Let $X = \{X(t), t\in\R^{N} \}$ be a centered Gaussian random field
with stationary
increments and $X(0) = 0$. Suppose conditions \textup{(H1)}, \textup{(H2)} and \textup{(H3$'$)} hold. Then
%
\begin{eqnarray}
\label{Eq:MEC} %
\E\bigl\{\varphi(A_u)\bigr\}&=& \sum
_{\{t\}\in\partial_0 T} \P\bigl(X(t)\geq u, \nabla X(t) \in E\bigl(\{t
\}\bigr)\bigr) + \sum^N_{k=1}\sum
_{J\in\partial_k T} \frac{1}{(2\pi)^{k/2}|\La
_J|^{1/2}}\nonumber
\\
&&{} \times\int_J \,dt\int^\infty_u
\,dx\int\cdots\int_{E(J)} \,dy_{J_1}\cdots
\,dy_{J_{N-k}}\frac{|\La_J-\La_J(t)|}{\gamma^k_t}
\nonumber
\\[-8pt]
\\[-8pt]
\nonumber
&&{} \times H_k \biggl(\frac{x}{\gamma_t} + \gamma_tC_{J_1}(t)y_{J_1}
+ \cdots+ \gamma_tC_{J_{N-k}}(t)y_{J_{N-k}} \biggr)
\\
& &{}\times p_{X(t), X_{J_1}(t), \ldots, X_{J_{N-k}}(t)}\bigl(x, y_{J_1}, \ldots, y_{J_{N-k}}|
\nabla X_{|J}(t)=0\bigr). \nonumber %
\end{eqnarray}
\end{theorem}

\begin{remark} If $Z = \{Z(t), t\in\R^{N} \}$ is a smooth centered
stationary Gaussian random field, then
the mean Euler characteristic of the excursion set $\{t \in T\dvtx Z(t)
\ge
u\}$ is given by Theorem~11.7.2 in \citet{AT07}. Applying Theorem~\ref{Thm:MEC} to $X(t) = Z(t) - Z(0)$,
(\ref{Eq:MEC}) computes the mean Euler
characteristic of $A_u = \{t \in T\dvtx Z(t)- Z(0) \ge u\}$.
\end{remark}

\begin{pf*}{Proof of Theorem~\ref{Thm:MEC}}
According to Corollary~11.3.2 in \citet{AT07}, (H1) and (H3$'$)
imply that $X$ is a Morse function a.s. It follows from (\ref{Eq:def of
Euler charac}) that
%
\begin{equation}
\label{Eq:mean EC} \E\bigl\{\varphi(A_u)\bigr\}=\sum
^N_{k=0}\sum_{J\in\partial_k T}(-1)^k
\E \Biggl\{ \sum^k_{i=0}
(-1)^i \mu_i(J) \Biggr\}.
\end{equation}
If $J \in\partial_0 T$, say $J= \{t\}$, it turns out that $\E\{\mu_0
(J)\}= \P (X(t)\geq u,
\nabla X(t) \in E(\{t\})  )$. If $J \in\partial_k T$ with $k \geq
1$, we apply the Kac--Rice
metatheorem in \citet{AT07} to obtain that the expectation
on the right-hand side of (\ref{Eq:mean EC}) becomes
%
\begin{eqnarray}
\label{Eq:mean mu 2} %
&& \int_J p_{\nabla X_{|J}(t)}(0) \,dt\nonumber\\
&&\quad{}\times \sum^k_{i=0} (-1)^i \E\bigl\{
\bigl|\operatorname {det} \nabla^2 X_{|J}(t)\bigr| \mathbh{1}_{\{\nabla^2 X_{|J}(t)
\in\mathcal{D}_i\}}
\mathbh{1}_{\{ (X_{J_1}(t), \ldots,
X_{J_{N-k}}(t))\in E(J)\}}\nonumber
\\
&&\hspace*{167pt}\qquad\quad{} \times\mathbh{1}_{\{
X(t)\geq u\}} | \nabla X_{|J}(t)=0 \bigr\}\nonumber
\\
& &\qquad= \frac{1}{(2\pi)^{k/2}|\La_J|^{1/2}}\int_J \,dt\int^\infty_u
\,dx\int\cdots\int_{E(J)} \,dy_{J_1}\cdots
\,dy_{J_{N-k}}
\\
&&\qquad\quad{} \times\E\bigl\{\operatorname{det} \nabla^2 X_{|J}(t) | X(t)=x,
X_{J_1}(t)=y_{J_1}, \ldots, X_{J_{N-k}}(t)=y_{J_{N-k}},\nonumber\\
&&\hspace*{270pt}
\nabla X_{|J}(t)=0 \bigr\}\nonumber
\\
&&\qquad\quad{} \times p_{X(t), X_{J_1}(t), \ldots, X_{J_{N-k}}(t)}\bigl(x, y_{J_1}, \ldots, y_{J_{N-k}}|
\nabla X_{|J}(t)=0\bigr).\nonumber %
\end{eqnarray}
For fixed $t$, let $Q_t$ be the positive definite matrix in (\ref
{Eq:decomposition of positivedef matrix}).
Then, similar to the proof in Lemma~\ref{Lem:MEC}, we can write
\begin{eqnarray*}
&&\E\bigl\{\operatorname{det} \bigl(Q_t \nabla^2
X_{|J}(t) Q_t\bigr) | X(t)=x, X_{J_1}(t)=y_{J_1},
\ldots, X_{J_{N-k}}=y_{J_{N-k}},\\
&&\hspace*{244pt} \nabla X_{|J}(t)=0 \bigr
\} %
\end{eqnarray*}
as $\E\{\operatorname{det} \overline{\Delta}(t, x) \}$, where $\overline
{\Delta
}(t, x)$ is a matrix consisting of Gaussian entries
$\overline{\Delta}_{ij}(t, x)$ with mean
%
\begin{eqnarray}
\label{Eq:expectation of Delta 2} %
&&\E\bigl\{\bigl(Q_t \nabla^2
X_{|J}(t) Q_t\bigr)_{ij} | X(t)=x,
X_{J_1}(t)=y_{J_1}, \ldots, X_{J_{N-k}}=y_{J_{N-k}},\nonumber\\
&&\hspace*{238pt}\nabla X_{|J}(t)=0 \bigr\}\nonumber
\\
&&\qquad= (-\delta_{ij}, 0, \ldots, 0) \bigl( \operatorname{Cov} \bigl(X(t),
X_{J_1}(t), \ldots, X_{J_{N-k}}(t), \nabla X_{|J}(t)
\bigr) \bigr)^{-1}
\\
&&\qquad\quad{}\times(x, y_{J_1}, \ldots, y_{J_{N-k}},0, \ldots,
0)^T\nonumber
\\
&&\qquad= -\frac{\delta_{ij}}{\gamma^2_t}\bigl(x+\gamma^2_tC_{J_1}(t)y_{J_1}
+ \cdots+ \gamma^2_tC_{J_{N-k}}(t)y_{J_{N-k}}
\bigr),\nonumber %
\end{eqnarray}
and covariance
\begin{eqnarray*}
&&\E\bigl\{ \bigl(\overline{\Delta}_{ij} (t, x)-\E\bigl\{\overline{
\Delta}_{ij} (t, x)\bigr\} \bigr) \bigl(\overline{\Delta}_{kl}(t,
x)-\E\bigl\{\overline{\Delta}_{kl} (t, x)\bigr\} \bigr)\bigr\} \\
&&\qquad=
\mathcal{S}_t(i,j,k,l)-\frac{\delta_{ij}\delta_{kl}}{\gamma^2_t}.
\end{eqnarray*}
Following the same procedure in the proof of Lemma~\ref{Lem:MEC}, we
obtain that the last
conditional expectation in (\ref{Eq:mean mu 2}) is equal to
%
\begin{eqnarray}
\label{Eq:formula for det 2}&& \frac{(-1)^k|\La_{J}-\La_{J}(t)|}{\gamma_t^k}
\nonumber
\\[-8pt]
\\[-8pt]
\nonumber
&&\qquad{}\times
H_k \biggl(\frac
{x}{\gamma_t}+
\gamma_tC_{J_1}(t)y_{J_1} + \cdots+ \gamma
_tC_{J_{N-k}}(t)y_{J_{N-k}} \biggr).
\end{eqnarray}
Plugging this into (\ref{Eq:mean mu 2}) and (\ref{Eq:mean EC}) yields
the desired result.
\end{pf*}

\begin{remark} Usually, for a nonstationary (including
constant-variance) Gaussian field $X$ on $\R^N$,
its mean Euler characteristic involves at least the third-order
derivatives of the covariance function.
For Gaussian random fields with stationary increments, as shown in
Lemma~\ref{Lem:independent},
$\E\{X_{ij}(t)X_k(t)\}=0$ and $\E\{X_{ij}(t)X_{kl}(t)\}$ is symmetric
in $i$, $j$, $k$, $l$. Hence,
the mean Euler characteristic becomes simpler, containing only up to
second-order derivatives
of the covariance function. This can also be seen from the spectral
representation \eqref{Eq:rep}
which implies that $\nabla X$ and $\nabla^2 X$ are stationary. In
various practical applications,
(\ref{Eq:MEC}) can be simplified with only an exponentially smaller
difference. See the discussions
in Section~\ref{sec5}.
\end{remark}

\section{Excursion probability}\label{sec4}
\subsection{Preliminaries}\label{sec4.1}
As in Section~\ref{sec3.1}, we decompose $T$ into its faces as $T= \bigcup_{k=0}^N\bigcup_{J\in\partial_k T}J$.
For $k \ge1$ and any $J\in\partial_k T$, define the number of \emph
{extended outward maxima}
above level $u$ as
\begin{eqnarray*}
M_u^E (J) & := &\# \bigl\{ t\in J\dvtx X(t)
\geq u, \nabla X_{|J}(t)=0, \operatorname {index} \bigl(\nabla^2
X_{|J}(t)\bigr)=k,
\\
&& \hspace*{118pt}\varepsilon^*_jX_j(t) \geq0 \mbox{ for all} j\notin
\sigma(J) \bigr\}. %
\end{eqnarray*}
In fact, $M_u^E (J)$ is the same as $\mu_k (J)$ defined in (\ref
{Eq:mu_iJ}) with $i= k$. For $k=0$ and any
$\{t\} \in\partial_0 T$, let
\[
M_u^E \bigl(\{t\}\bigr) := \mathbh{1}_{\{X(t)\geq u, \nabla X(t) \in E(\{t\})
\}}=
\mathbh{1}_{\{X(t)\geq u,
\varepsilon^*_j X_j(t) \ge0,\forall j = 1, \ldots, N\}}.
\]
One can show easily that, under conditions (H1) and
(H3$'$),
\[
\Bigl\{\sup_{t\in T} X(t) \geq u \Bigr\} = \bigcup
_{k=0}^N\bigcup_{J\in
\partial_k T}
\bigl\{M_u^E (J) \geq1 \bigr\} \qquad\mbox{ a.s.}
\]
It follows that
%
\begin{equation}
\label{Ineq:upperbound} \quad\P \Bigl\{\sup_{t\in T} X(t) \geq u \Bigr\} \leq
\sum_{k=0}^N\sum
_{J\in
\partial_k T} \P\bigl\{M_u^E (J) \geq1 \bigr
\} \leq\sum_{k=0}^N\sum
_{J\in\partial_k T} \E\bigl\{M_u^E (J) \bigr\}.
\end{equation}
On the other hand, by the Bonferroni inequality,
\[
\P \Bigl\{\sup_{t\in T} X(t) \geq u \Bigr\} \geq\sum
_{k=0}^N\sum_{J\in
\partial_k T} \P
\bigl\{M_u^E (J) \geq1 \bigr\} - \sum
_{J\neq J'} \P\bigl\{M_u^E (J) \geq1,
M_u^E \bigl(J'\bigr) \geq1 \bigr\}.
\]
Note that [cf. \citet{Piterbarg96}]
\[
\E\bigl\{M_u^E (J) \bigr\} - \P\bigl
\{M_u^E (J) \geq1 \bigr\} \le\tfrac{1}{2}\E\bigl
\{M_u^E (J) \bigl(M_u^E (J) -1
\bigr)\bigr\},
\]
together with the obvious bound $\P\{M_u^E (J) \geq1 , M_u^E (J')
\geq
1 \} \leq\E\{M_u^E (J)\times  M_u^E (J') \}$,
we obtain the following lower bound for the excursion probability
%
\begin{eqnarray}
\label{Ineq:lowerbound} %
\P \Bigl\{\sup_{t\in T} X(t) \geq u
\Bigr\} & \geq&\sum_{k=0}^N\sum
_{J\in\partial_k T} \biggl( \E\bigl\{M_u^E (J)
\bigr\} - \frac{1}{2}\E\bigl\{M_u^E (J)
\bigl(M_u^E (J) -1\bigr) \bigr\} \biggr)
\nonumber
\\[-8pt]
\\[-8pt]
\nonumber
&&{} - \sum_{J\neq J'} \E\bigl\{M_u^E
(J) M_u^E \bigl(J'\bigr) \bigr\}.
\end{eqnarray}
Define the number of \emph{local maxima} above level $u$ as
\[
M_u (J) := \# \bigl\{ t\in J\dvtx X(t)\geq u, \nabla
X_{|J}(t)=0, \operatorname{ index} \bigl(\nabla^2
X_{|J}(t)\bigr)=k \bigr\}.
\]
Then obviously $M_u (J) \geq M^E_u (J)$ and $M_u (J)$ is the same as
$\wt{\mu}_k (J)$ defined in (\ref{Eq:mu_iJ2})
with $i= k$. It follows similarly that
%
\begin{eqnarray}
\label{Ineq:lowerbound and upperbound} %
\sum_{k=0}^N
\sum_{J\in\partial_k T} \E\bigl\{M_u (J) \bigr\} &\geq&
\P \Bigl\{\sup_{t\in T} X(t) \geq u \Bigr\}
\nonumber
\\
&\geq&\sum_{k=0}^N\sum
_{J\in\partial_k T} \biggl( \E\bigl\{M_u (J) \bigr\}-
\frac
{1}{2}\E\bigl\{M_u (J) \bigl(M_u (J) -1
\bigr)\bigr\} \biggr) \\
&&{}- \sum_{J\neq J'} \E\bigl
\{M_u (J) M_u \bigl(J'\bigr) \bigr\}.\nonumber
\end{eqnarray}

We will use (\ref{Ineq:upperbound}) and (\ref{Ineq:lowerbound}) to
estimate the excursion
probability for the general case in Theorem~\ref{Thm:MEC approximation 2}.
Inequalities in (\ref{Ineq:lowerbound and upperbound}) provide another
method to
approximate the excursion probability in some special cases; see Theorem~\ref{Thm:MEC approximation 1}. The advantage of (\ref{Ineq:lowerbound
and upperbound})
is that the principal term induced by $\sum_{k=0}^N\sum_{J\in
\partial_k
T} \E\{M_u (J)\}$
is much easier to compute compared with the one induced by $\sum_{k=0}^N\sum_{J\in\partial_k T} \E\{M_u^E (J) \}$.

\subsection{Estimating the moments: Major terms and error terms}\label{sec4.2}

The following two lemmas provide the estimations for the principal
terms in approximating the excursion probability.

\begin{lemma}\label{Lem:mean M_u}
Let $X$ be a Gaussian field as in Theorem~\ref{Thm:MEC}. Then for each
$J\in\partial_k T$ with $k\geq1$, there exists some constant $\alpha
>0$ such that
%
\begin{eqnarray}
\label{Eq:mean M_uJ} %
&&\E\bigl\{M_u(J)\bigr\} \nonumber\\
&&\qquad=
\frac{1}{(2\pi)^{(k+1)/2}|\La_J|^{1/2}}\\
&&\qquad\quad{}\times\int_{J} \frac
{|\La_J-\La_J(t)|}{\theta^k_t}
H_{k-1} \biggl(\frac{u}{\theta_t} \biggr)e^{-u^2/{(2\theta^2_t)}} \,dt\bigl(1+ o
\bigl(e^{-\alpha u^2}\bigr)\bigr).\nonumber %
\end{eqnarray}
\end{lemma}

\begin{pf} Following the notation in the proof of Lemma~\ref
{Lem:MEC}, we obtain similarly that
%
\begin{eqnarray}
\label{Eq:mean mu 3} %
&&\E\bigl\{M_u (J)\bigr\}\nonumber\\
&&\qquad = \int
_{J} p_{\nabla X_{|J}(t)}(0) \,dt \E\bigl\{ \bigl|\operatorname {det}
\nabla^2 X_{|J}(t)\bigr| \mathbh{1}_{\{\nabla^2 X_{|J}(t) \in\mathcal{D}_k\}}
\mathbh{1}_{\{
X(t)\geq u\}} | \nabla X_{|J}(t)=0 \bigr\}\hspace*{-20pt}
\nonumber
\\[-8pt]
\\[-8pt]
\nonumber
&&\qquad=\int_{J} \,dt\int^\infty_u
\,dx \frac{(-1)^ke^{-x^2/(2\theta
^2_t)}}{(2\pi
)^{(k+1)/2}|\La_J|^{1/2}\theta_t}\\
&&\qquad\quad{}\times \E\bigl\{\operatorname{det} \nabla^2
X_{|J}(t) \mathbh{1}_{\{\nabla^2 X_{|J}(t)
\in\mathcal{D}_k\}} | X(t)=x, \nabla
X_{|J}(t)=0 \bigr\}.\nonumber %
\end{eqnarray}
Recall that $Q_t$ is the $k\times k$ positive definite matrix in (\ref
{Eq:decomposition of positivedef matrix}).
We write (\ref{Eq:expectation of Delta}) as
\[
\E\bigl\{Q_t \nabla^2 X_{|J}(t)
Q_t | X(t)=x, \nabla X_{|J}(t)=0 \bigr\} = -
\frac
{x}{\theta^2_t}I_k.
\]
Make change of variables
\[
V(t) = Q_t\nabla^2 X_{|J}(t)Q_t
+ \frac{x}{\theta^2_t}I_k.
\]
Then $(V(t)|X(t)=x, \nabla X_{|J}(t)=0)$ is a Gaussian matrix
whose mean is 0 and covariance is the same as that of $(Q_t \nabla^2
X_{|J}(t) Q_t | X(t)=x, \nabla X_{|J}(t)=0)$.
Write $V(t) = (V_{ij}(t))_{1\leq i, j\leq k}$ and denote the density of
Gaussian vectors $((V_{ij}(t))_{1\leq i\leq j\leq k}|X(t)=x, \nabla
X_{|J}(t)=0)$
by $h_t(v)$, where $v=(v_{ij})_{1\leq i\leq j\leq k}\in\R
^{k(k+1)/2}$. Then
%
\begin{eqnarray}
\label{Eq:expectation of det} %
&&\E \bigl\{\operatorname{det} \bigl(Q_t
\nabla^2 X_{|J}(t)Q_t\bigr)
\mathbh{1}_{\{\nabla^2
X_{|J}(t) \in\mathcal{D}_k\}} | X(t)=x, \nabla X_{|J}(t)=0 \bigr\}
\nonumber\\
&&\qquad = \E\bigl\{\operatorname{det} \bigl(Q_t\nabla^2
X_{|J}(t)Q_t\bigr) \mathbh{1}_{\{
Q_t\nabla
^2 X_{|J}(t)Q_t \in\mathcal{D}_k\}} | X(t)=x,
\nabla X_{|J}(t)=0 \bigr\}
\\
&&\qquad = \int_{\{v\dvtx(v_{ij})-({x}/{\theta^2_t})I_k \in\mathcal
{D}_k\}} \operatorname{det} \biggl((v_{ij})-
\frac{x}{\theta^2_t} I_k \biggr) h_t(v) \,dv, \nonumber%
\end{eqnarray}
where $(v_{ij})$ is the abbreviation for the matrix $(v_{ij})_{1\leq i,
j\leq k}$.
Since $\{\theta^2_t\dvtx t\in T\}$ is bounded, there exists a constant
$c>0$ such that
\[
(v_{ij})-\frac{x}{\theta^2_t}I_k \in\mathcal{D}_k\qquad
\forall\bigl\| (v_{ij})\bigr\|:= \Biggl(\sum_{i,j=1}^k
v_{ij}^2 \Biggr)^{1/2} <\frac{x}{c}.
\]
Thus, we can write (\ref{Eq:expectation of det}) as
%
\begin{eqnarray}
\label{Eq:expectation of det 2} %
&&\int_{\R^{k(k+1)/2}} \operatorname{det}
\biggl((v_{ij})-\frac{x}{\theta^2_t}I_k \biggr)
h_t(v) \,dv \nonumber\\
&&\quad{}- \int_{\{v\dvtx(v_{ij})-
({x}/{\theta^2_t})I_k \notin\mathcal{D}_k\}} \operatorname{det}
\biggl((v_{ij})-\frac{x}{\theta^2_t}I_k \biggr)
h_t(v) \,dv
\\
&&\qquad = \E\bigl\{\operatorname{det} \bigl(Q_t\nabla^2
X_{|J}(t)Q_t\bigr) | X(t)=x, \nabla X_{|J}(t)=0
\bigr\} + Z(t,x),\nonumber %
\end{eqnarray}
where $Z(t,x)$ is the second integral in the first line of (\ref
{Eq:expectation of det 2}) and it satisfies
\[
\bigl|Z(t,x)\bigr| \leq\int_{\|(v_{ij})\|\geq{x}/{c}} \biggl|\operatorname{det}
\biggl((v_{ij})-
\frac{x}{\theta^2_t}I_k \biggr) \biggr| h_t(v) \,dv.
\]
Denote by $G(t)$ the covariance matrix of $((V_{ij}(t))_{1\leq i\leq
j\leq k}|X(t)=x, \nabla X_{|J}(t)=0)$. Then
by Lemma~\ref{Lem:eigenvalue} in the \hyperref[app]{Appendix}, the
eigenvalues of
$G(t)$ and those of $(G(t))^{-1}$ are
bounded for all $t\in T$. It follows that there exists some constant
$\alpha'>0$ such that $h_t(v) =
o(e^{-\alpha' \|(v_{ij})\|^2})$ and hence $|Z(t,x)| = o(e^{-\alpha
x^2})$ for some constant $\alpha>0$ uniformly
for all $t\in T$. Combining this with (\ref{Eq:mean mu 3}), (\ref
{Eq:expectation of det}), (\ref{Eq:expectation of det 2})
and the proof of Lemma~\ref{Lem:MEC} yields (\ref{Eq:mean M_uJ}).
\end{pf}

\begin{lemma}\label{Lem:mean M_u^E}
Let $X$ be a Gaussian field as in Theorem~\ref{Thm:MEC}. Then for each
$J\in\partial_k T$ with $k\geq1$, there exists some constant $\alpha
>0$ such that
%
\begin{eqnarray}
\label{Eq:mean M_u^E(J)} %
&&\E\bigl\{M_u^E(J)\bigr\} \nonumber\\
&&\qquad=
\frac{1}{(2\pi)^{k/2}|\La_J|^{1/2}} \int_J \,dt\int^\infty_u
\,dx\int\cdots\int_{E(J)} \,dy_{J_1}\cdots
\,dy_{J_{N-k}}
\nonumber
\\[-8pt]
\\[-8pt]
\nonumber
&&\qquad\quad{} \times\frac{|\La_J-\La_J(t)|}{\gamma^k_t}H_k \biggl(\frac
{x}{\gamma_t} +
\gamma_tC_{J_1}(t)y_{J_1} + \cdots+ \gamma
_tC_{J_{N-k}}(t)y_{J_{N-k}} \biggr)
\\
& &\qquad\quad{}\times p_{X(t), X_{J_1}(t), \ldots, X_{J_{N-k}}(t)}\bigl(x, y_{J_1}, \ldots, y_{J_{N-k}}|
\nabla X_{|J}(t)=0\bigr) \bigl(1+ o\bigl(e^{-\alpha u^2}\bigr)\bigr).\nonumber
\end{eqnarray}
\end{lemma}

\begin{pf}
Similar to the proof in Theorem~\ref{Thm:MEC}, we see
that $\E\{M_u^E(J)\}$ is equal to
\begin{eqnarray*}
&&\int_J \frac{(-1)^k|\La_J-\La_J(t)|}{(2\pi)^{k/2}|\La_J|^{1/2}} \,dt\int
^\infty_u \,dx \int\cdots\int_{E(J)}
\,dy_{J_1}\cdots \,dy_{J_{N-k}}\\
&&\quad{}\times \E\bigl\{ \operatorname{det}
\bigl(Q_t\nabla^2 X_{|J}(t)Q_t
\bigr)
\mathbh{1}_{\{Q_t\nabla^2 X_{|J}(t)Q_t \in\mathcal
{D}_k\}} |\\
&&\hspace*{14pt}\qquad X(t)=x, X_{J_1}(t)=y_{J_1},
\ldots, X_{J_{N-k}}(t)=y_{J_{N-k}}, \nabla X_{|J}(t)=0
\bigr\}
\\
&&\quad{} \times p_{X(t), X_{J_1}(t), \ldots, X_{J_{N-k}}(t)}\bigl(x, y_{J_1}, \ldots, y_{J_{N-k}}|
\nabla X_{|J}(t)=0\bigr)
\\
&&\qquad:= \int_J \frac{(-1)^k|\La_J-\La_J(t)|}{(2\pi)^{k/2}|\La_J|^{1/2}} \,dt\int^\infty_u
\,dx K(t,x), %
\end{eqnarray*}
where $Q_t$ is the positive definite matrix in (\ref{Eq:decomposition
of positivedef matrix}). Then, by using a
similar argument as in the proof of Lemma~\ref{Lem:mean M_u} to
estimate $ K(t,x)$, we obtain the desired result.
\end{pf}

We call a function $h(u)$ \emph{super-exponentially small} [when
compared with $\P(\sup_{t\in T} X(t) \geq u )$],
if there exists a constant $\alpha>0$ such that $h(u) =\break  o(e^{-\alpha
u^2 - u^2/{(2\sigma_T^2)}})$ as $u \to\infty$.

The following lemma is Lemma~4 in \citet{Piterbarg96}. It will be
used to
show that the factorial moments of
$M_u (J)$ and $M_u ^E(J)$ are usually super-exponentially small.

\begin{lemma} \label{Lem:Piterbarg}
Let $\{X(t)\dvtx t\in\R^N\}$ be a centered Gaussian field satisfying
\textup{(H1)} and \textup{(H3)}.
Then for any $\ep>0$, there exists $\varepsilon_1 >0$ such that for
any $J\in\partial_k T$ and $u$ large enough,
\[
\E\bigl\{M_u (J) \bigl(M_u(J)- 1\bigr)\bigr
\} \leq e^{-u^2/(2\beta_J^2 +\varepsilon)} + e^{-u^2/(2\sigma_J^2 -\ep_1)}, %
\]
where $\sigma_J^2=\sup_{t\in J} \operatorname{ Var} (X(t))$ and $\beta_J^2 =
\sup_{t\in J} \sup_{e\in\mathbb{S}^{k-1}}
\operatorname{ Var} (X(t)|\nabla X_{|J}(t),\break  \nabla^2 X_{|J}(t)e)$. Here and in
the sequel, $\mathbb{S}^{k-1}$
is the unit sphere in $\R^{k}$.
\end{lemma}

\begin{corollary}\label{Cor:Piterbarg}
Let $X = \{X(t), t\in\R^{N} \}$ be a centered Gaussian random field
with stationary increments satisfying
\textup{(H1)}, \textup{(H2)} and \textup{(H3)}. Then for all
$J\in\partial
_k T$, $\E\{M_u (J)(M_u (J) -1)\}$
and $\E\{M_u^E (J)(M_u^E (J) -1) \}$ are super-exponentially small.
\end{corollary}

\begin{pf} Since $M_u^E (J) \leq M_u (J)$, we only need to show
that $\E\{M_u (J)\times (M_u (J) -1) \}$
is super-exponentially small. If $k=0$, then $M_u (J)$ is either 0 or 1
and hence $\E\{M_u (J)(M_u (J) -1) \}=0$. If $k\geq1$, then, thanks
to Lemma~\ref{Lem:Piterbarg}, it suffices to show that $\beta_J^2$ is
strictly less than $\sigma_T^2$.

Clearly, $\operatorname{Var} (X(t)|\nabla X_{|J}(t), \nabla^2 X_{|J}(t)e)
\leq
\sigma_T^2$ for every $e \in\mathbb{S}^{k-1}$ and $t \in T$.
On the other hand,
%
\begin{equation}
\label{Eq:contra} 
\operatorname{Var} \bigl(X(t)|\nabla X_{|J}(t),
\nabla^2 X_{|J}(t)e\bigr) = \sigma_T^2
\quad\Longrightarrow\quad\E\bigl\{X(t) \bigl(\nabla^2 X_{|J}(t)e
\bigr)\bigr\}=0.\hspace*{-30pt} %
\end{equation}
Note that, by (\ref{ladiff}), the right-hand side of (\ref{Eq:contra})
is equivalent to $(\La_J(t) - \La_J)e=0$. However,
by (H2), $\La_J(t) - \La_J$ is negative definite, which implies
$(\La_J(t) - \La_J)e\neq0$ for all $e \in\mathbb{S}^{k-1}$. Thus
\[
\operatorname{Var} \bigl(X(t)|\nabla X_{|J}(t), \nabla^2
X_{|J}(t)e\bigr) < \sigma_T^2
\]
for all $e\in\mathbb{S}^{k-1}$ and $t \in T$. This and the continuity
of $\operatorname{Var} (X(t)|\nabla X_{|J}(t),\break  \nabla^2 X_{|J}(t)e)$
in $(e, t)$ imply $\beta_J^2 < \sigma_T^2$.
\end{pf}

The following lemma shows that the cross terms in (\ref
{Ineq:lowerbound}) and (\ref{Ineq:lowerbound and upperbound})
are super-exponentially small if the two faces are not adjacent. For
the case when the faces are adjacent, the proof is more technical.
See the proofs in Theorems \ref{Thm:MEC approximation 1} and \ref
{Thm:MEC approximation 2}.
%
\begin{lemma}\label{Lem:disjoint faces}
Let $X = \{X(t), t\in\R^{N} \}$ be a centered Gaussian random field
with stationary increments satisfying
\textup{(H1)} and \textup{(H3)}. Let $J$ and $J'$ be two faces of $T$
such that their distance is positive,
that is, $\inf_{t\in J, s\in J'}\|s-t\|>\delta_0$ for some
$\delta_0>0$. Then $\E\{M_u (J) M_u (J') \}$ is super-exponentially small.
\end{lemma}

\begin{pf} We first consider the case when $\operatorname{ dim}(J) =k \geq
1$ and
$\operatorname{ dim}(J') =k' \geq1$. By the Kac--Rice metatheorem for higher
moments [the proof is the same as that of Theorem~11.5.1 in \citet{AT07}],
%
\begin{eqnarray}
\label{Eq:disjoint faces} %
&&\E\bigl\{M_u (J) M_u
\bigl(J'\bigr) \bigr\} \nonumber\\
&&\qquad= \int_{J} \,dt \int
_{J'} \,ds \E\bigl\{ \bigl|\operatorname{ det} \nabla^2
X_{|J}(t)\bigr | \bigl|\operatorname{ det} \nabla^2 X_{|J'}(s) \bigr|\nonumber\\
&&\hspace*{76pt}\qquad{}\times
\mathbh {1}_{\{X(t)\geq u, X(s)\geq u\}}\nonumber
\\
& &\hspace*{64pt}\qquad\quad{}\times\mathbh{1}_{\{\nabla^2 X_{|J}(t)
\in\mathcal{D}_k,
\nabla^2 X_{|J'}(s) \in\mathcal{D}_{k'}\}} |\nonumber\\
&&\hspace*{100pt} X(t)=x, X(s)=y, \nabla X_{|J}(t)=0,\nabla X_{|J'}(s)=0\bigr\}\nonumber
\\
&&\hspace*{86pt}{}\times p_{X(t), X(s),
\nabla X_{|J}(t), \nabla
X_{|J'} (s)}(x,y,0,0)
\\
&&\qquad\leq\int_{J} \,dt \int_{J'} \,ds \int
_u^\infty \,dx \int_u^\infty\nonumber
\,dy \\
&&\qquad\quad{}\E \bigl\{ \bigl|\operatorname{ det} \nabla^2 X_{|J}(t) \bigr|
\bigl|\operatorname{ det}
\nabla^2 X_{|J'}(s)\bigr ||
\nonumber\\
&&\hspace*{11pt}\qquad\quad {}\times X(t)=x, X(s)=y, \nabla X_{|J}(t)=0, \nabla X_{|J'}(s)=0
\bigr\} p_{X(t), X(s)}(x,y)
\nonumber\\
&&\qquad\quad \times p_{\nabla X_{|J}(t),
 \nabla X_{|J'} (s)}\bigl(0,0|X(t)=x, X(s)=y\bigr).\nonumber %
\end{eqnarray}
Note that the following two inequalities hold: For constants $a_i$ and $b_j$,
\[
\prod_{i=1}^k |a_i| \prod
_{j=1}^{k'} |b_j|\leq
\frac{1}{k+k'} \Biggl(\sum_{i=1}^k
|a_i|^{k+k'} + \sum_{j=1}^{k'}
|b_j|^{k+k'} \Biggr);
\]
and for any Gaussian variable $\xi$ and positive integer $l$,
\[
\E|\xi|^l \leq\E\bigl(|\E\xi|+|\xi-\E\xi|\bigr)^l
\leq2^l \bigl(|\E\xi|^l + \E|\xi -\E\xi|^l
\bigr) = 2^l \bigl(|\E\xi|^l + K_l\bigl(\operatorname{
Var}(\xi)\bigr)^{l/2}\bigr),
\]
where the constant $K_l$ depends only on $l$. It follows from these two
inequalities that there
exist some positive constants $C_1$ and $N_1$ such that for large $x$
and $y$,
%
\begin{eqnarray}
\label{Eq:disjoint faces 2} %
&&\sup_{t\in J, s\in J'}\E\bigl\{ \bigl|\operatorname{ det}
\nabla^2 X_{|J}(t)\bigr | \bigl|\operatorname{ det} \nabla^2
X_{|J'}(s) \bigr| | X(t)=x, X(s)=y,
\nabla X_{|J}(t)=0,\nonumber\\
&& \hspace*{273pt}\nabla X_{|J'}(s)=0\bigr\}\\
&&\qquad \leq
C_1x^{N_1}y^{N_1}.\nonumber %
\end{eqnarray}
Also, there exists a positive constant $C_2$ such that
%
\begin{eqnarray}
\label{Eq:disjoint faces 3} %
&&\sup_{t\in J, s\in J'} p_{\nabla X_{|J}(t), \nabla X_{|J'}
(s)}
\bigl(0,0|X(t)=x, X(s)=y\bigr)\nonumber
\\
&&\qquad \leq\sup_{t\in J, s\in J'}(2\pi)^{-(k+k')/2}
\nonumber
\\[-8pt]
\\[-8pt]
\nonumber
&&\qquad\quad{}\times\bigl[\operatorname{ det Cov}
\bigl(\nabla X_{|J}(t), \nabla X_{|J'} (s) | X(t)=x, X(s)=y
\bigr)\bigr]^{-1/2} \\
&&\qquad\leq C_2.\nonumber %
\end{eqnarray}
Let $\rho(\delta_0)= \sup_{\|s-t\|>\delta_0} \frac{|\E\{X(t)X(s)\}
|}{\sqrt{\nu(t)\nu(s)}}$ which is
strictly less than 1 \,due to (H3), then $\forall\ep>0$, there
exists a positive constant
$C_3$ such that for all $t\in J$, $s\in J'$ and $u$ large enough,
%
\begin{eqnarray}
\label{Eq:disjoint faces 4} %
&&\int_u^\infty\int
_u^\infty x^{N_1}y^{N_1}p_{X(t), X(s)}(x,y)
\,dx\,dy \nonumber\\
&&\qquad= \E\bigl\{ \bigl[X(t)X(s)\bigr]^{N_1} \mathbh{1}_{\{X(t)\geq u, X(s)\geq u\}}
\bigr\}
\nonumber
\\[-8pt]
\\[-8pt]
\nonumber
&&\qquad \leq\E\bigl\{ \bigl[X(t)+X(s)\bigr]^{2N_1} \mathbh{1}_{\{X(t)+ X(s)\geq2u\}
}
\bigr\} \\
&&\qquad\leq C_3 \exp \biggl(\ep u^2-\frac{u^2}{(1+\rho(\delta_0))\sigma
_T^2}
\biggr).\nonumber %
\end{eqnarray}
Combining (\ref{Eq:disjoint faces}) with (\ref{Eq:disjoint faces 2}),
(\ref{Eq:disjoint faces 3}) and
(\ref{Eq:disjoint faces 4}) yields that $\E\{M_u (J) M_u (J') \}$ is
super-exponentially small.

When only one of the faces, say $J$, is a singleton, then let $J=\{t_0\}
$ and we have
%
\begin{eqnarray}
\label{Eq:single point} %
&&\E\bigl\{M_u (J) M_u
\bigl(J'\bigr) \bigr\} \nonumber\\
&&\qquad\leq\int_{J'} \,ds \int
_u^\infty \,dx \int_u^\infty
\,dy p_{X(t_0), X(s), \nabla X_{|J'} (s)}(x,y,0)
\\
&&\qquad\quad{} \times\E\bigl\{\bigl |\operatorname{ det} \nabla^2 X_{|J'}(s) \bigr| |
X(t_0)=x, X(s)=y, \nabla X_{|J'}(s)=0 \bigr\}.\nonumber %
\end{eqnarray}
Following the previous discussion yields that $\E\{M_u (J) M_u (J')\}$
is super-exponentially small.

Finally, if both $J$ and $J'$ are singletons, then $\E\{M_u (J) M_u
(J') \}$ becomes the joint
probability of two Gaussian variables exceeding level $u$, and hence is trivial.
\end{pf}

\subsection{Main results and their proofs}\label{sec4.3}

Now we are ready to prove our main results on approximating the
excursion probability
$\P \{\sup_{t\in T} X(t) \geq u  \}$. Theorem~\ref{Thm:MEC
approximation 1} contains
a mild technical condition (\ref{Condition:boundary max point}) which
specifies the way that
the variogram $\nu(t)$ attains its maximum on the boundary of $T$. In
particular, it implies
that, at each point on $\partial T$ where $\nu(t)$ achieves $\sigma
_T^2:=\sup_{t\in T}\nu(t)$,
$\nabla\nu(t)$ is not zero. In the case
of $N=1$ and $T=[a, b]$, if $\nu(t)$ attains its maximum $\sigma_T^2$
at the end point $a$ or $b$,
then (\ref{Condition:boundary max point}) requires that $\nu(t)$ is
strictly monotone in a
neighborhood of that end point. Notice that, if $\nu(t)$ only attains
its maximum in $\partial_N T$,
the interior of $T$, then (\ref{Condition:boundary max point}) is
satisfied automatically.
In this sense, (\ref{Condition:boundary max point}) is more general
than the corresponding
condition in Theorem~5 of Aza\"is and Wschebor (\citeyear{AzaisW08}).

\begin{theorem}\label{Thm:MEC approximation 1}
Let $X = \{X(t)\dvtx t\in\R^{N} \}$ be a centered Gaussian random field
with stationary increments
such that \textup{(H1)}, \textup{(H2)} and \textup{(H3)} are fulfilled.
Suppose that for any face $J$,
%
\begin{equation}
\label{Condition:boundary max point} \bigl\{t\in J\dvtx\nu(t) = \sigma_T^2,
\nu_j(t) = 0 \mbox{ for some } j\notin \sigma(J) \bigr\} = \varnothing.
\end{equation}
Then there exists some constant $\alpha>0$ such that
%
\begin{eqnarray}
\label{Eq:excursion probability for special case} %
&&\P \Bigl\{\sup_{t\in T} X(t) \geq u
\Bigr\}\nonumber\\
&&\qquad= \sum_{k=0}^N\sum
_{J\in
\partial_k T} \E\bigl\{M_u (J) \bigr\} + o
\bigl(e^{-\alpha u^2 - u^2/{(2\sigma_T^2)} }\bigr)
\nonumber
\\[-8pt]
\\[-8pt]
\nonumber
&&\qquad= \sum_{\{t\}\in\partial_0 T} \Psi \biggl(\frac{u}{\sqrt{\nu(t)}}
\biggr) + \sum_{k=1}^N\sum
_{J\in\partial_k T}\frac{1}{(2\pi)^{(k+1)/2}|\La
_J|^{1/2}}
\\
&&\qquad\quad{} \times\int_J \frac{|\La_J-\La_J(t)|}{\theta^k_t} H_{k-1}
\biggl(\frac{u}{\theta_t} \biggr) e^{-u^2/{(2\theta^2_t)}} \,dt + o\bigl(e^{-\alpha u^2 - u^2/{(2\sigma_T^2) }}
\bigr). \nonumber%
\end{eqnarray}
\end{theorem}

\begin{remark}
It should be mentioned that, for Gaussian random fields with constant
variance, Taylor, Takemura and Adler (\citeyear{TTA05}) provided
an explicit form for the constant $\alpha$ in (\ref{Eq:MEC approx
error}). However, in Theorems \ref{Thm:MEC approximation 1} and
\ref{Thm:MEC approximation 2}, we are not able to give explicit
information on the value(s) of $\alpha$.
\end{remark}

\begin{pf*}{Proof of Theorem~\ref{Thm:MEC approximation 1}} Since the
second equality in (\ref{Eq:excursion probability for special case})
follows from Lemma~\ref{Lem:mean M_u} directly, we only need to prove
the first one. By
(\ref{Ineq:lowerbound and upperbound}) and Corollary~\ref
{Cor:Piterbarg}, it suffices to
show that the last term in (\ref{Ineq:lowerbound and upperbound}) is
super-exponentially
small. Thanks to Lemma~\ref{Lem:disjoint faces},\vspace*{1pt} we only need to
consider the case when
the distance of $J$ and $J'$ is~0, that is, $I:= \bar{J}\cap\bar{J'}
\neq\varnothing$. Without
loss of generality, we assume
%
\begin{eqnarray}
\label{Eq:assumption for faces} %
\sigma(J)&=& \{1, \ldots, m, m+1, \ldots, k\},
\nonumber
\\[-8pt]
\\[-8pt]
\nonumber \sigma
\bigl(J'\bigr)&=& \bigl\{1, \ldots , m, k+1, \ldots,
k+k'-m\bigr\}, %
\end{eqnarray}
where $0 \leq m \leq k \leq k' \leq N$ and $k'\geq1$. Recall that, if
$k=0$, then $\sigma(J)
= \varnothing$. Under assumption (\ref{Eq:assumption for faces}), we have
$J\in\partial_k T$, $J'\in\partial_{k'} T$
and $\operatorname{dim} (I) =m$.

\textit{Case} 1: $k=0$, that is, $J$ is a singleton, say $J=\{t_0\}$. If $\nu
(t_0) < \sigma_T^2$, then
by (\ref{Eq:single point}), it is trivial to show that $\E\{M_u (J)
M_u (J')\}$ is
super-exponentially small. Now we consider the case $\nu(t_0) = \sigma
_T^2$. Due to
(\ref{Condition:boundary max point}), $\E\{X(t_0) X_1(t_0)\}\neq0$,
and hence by
continuity, there exists $\delta>0$ such that $\E\{X(s) X_1(s)\}\neq
0$ for all
$\|s-t_0\|\leq\delta$. It follows from (\ref{Eq:single point}) that
$\E\{M_u (J) M_u (J') \}$
is bounded from above by
\begin{eqnarray*}
&&\int_{s\in J'\dvtx\|s-t_0\|> \delta} \,ds \int_u^\infty
\,dx \int_u^\infty \,dy\\
&&\quad{}\times \E\bigl\{ \bigl|\operatorname{det}
\nabla^2 X_{|J'}(s) \bigr| | X(t_0)=x, X(s)=y,
\nabla X_{|J'}(s)=0 \bigr\}
\\
&&\quad{} \times p_{X(t_0),
X(s), \nabla X_{|J'} (s)}(x,y,0)
\\
&&\quad{}+ \int_{s\in J'\dvtx\|s-t_0\|\leq\delta} \,ds \int_u^\infty
\,dy \\
&&\qquad{}\times\E\bigl\{ \bigl|\operatorname{det} \nabla^2 X_{|J'}(s)\bigr | | X(s)=y,
\nabla X_{|J'}(s)=0 \bigr\} p_{X(s), \nabla X_{|J'} (s)}(y,0)
\\
&&\quad\qquad:= I_1+I_2. %
\end{eqnarray*}
Following the proof of Lemma~\ref{Lem:disjoint faces}, we can show that
$I_1$ is super-exponentially small.
Note that there exists a constant $\ep_0>0$ such that
\[
\sup_{s\in J'\dvtx\|s-t_0\|\leq\delta} \operatorname{ Var} \bigl(X(s)| \nabla X_{|J'}(s)
\bigr) \leq\sup_{s\in J'\dvtx\|s-t_0\|\leq\delta} \operatorname{ Var} \bigl(X(s)| X_1(s)
\bigr) \leq\sigma_T^2 - \ep_0.
\]
This implies that $I_2$, and hence $\E\{M_u (J) M_u (J') \}$ are
super-exponentially small.

\textit{Case} 2: $k\geq1$. For all $t\in I$ with $\nu(t) = \sigma_T^2$, by
assumption (\ref{Condition:boundary max point}),\break 
$\E\{X(t) X_i(t)\}\neq0$, $\forall i = m+1,\ldots, k+k'-m$. Since
$I$ is a compact set, we see that there exist constants
$\ep_1, \delta_1>0$ such that
%
\begin{eqnarray}
\label{Eq:super-exponentially small}
&&\sup_{t\in B, s\in B'} \operatorname{ Var} \bigl( X(t) |
X_{m+1}(t), \ldots, X_{k} (t), X_{k+1} (s), \ldots,
X_{k+k'-m} (s)\bigr)
\nonumber
\\[-8pt]
\\[-8pt]
\nonumber
&&\qquad\leq\sigma_T^2 -
\ep_1,
\end{eqnarray}
where $B = \{t\in J\dvtx\operatorname{dist}(t, I) \leq\delta_1\}$ and $B' =
\{
s\in J'\dvtx\operatorname{dist}(s, I) \leq\delta_1\}$.
It follows from (\ref{Eq:disjoint faces}) that $\E\{M_u (J) M_u (J')
\}
$ is bounded by
\begin{eqnarray*}
&&\int\!\!\int_{(J\times J') \setminus(B\times B')} \,dt\,ds \int_u^\infty
\,dx \int_u^\infty \,dy p_{X(t), X(s), \nabla X_{|J} (t), \nabla X_{|J'}
(s)}(x,y,0,0)
\\
&&\qquad{} \times\E\bigl\{ \bigl|\operatorname{det} \nabla^2 X_{|J}(t) \bigr|\bigl|
\operatorname{det} \nabla ^2 X_{|J'}(s) \bigr| | X(t)=x, X(s)=y, \nabla
X_{|J}(t)=0, \\
&&\hspace*{273pt}\nabla X_{|J'}(s)=0 \bigr\}
\\
&&\qquad{}+ \int\!\!\int_{ B\times B'} \,dt\,ds \int_u^\infty
\,dx p_{X(t)}\bigl(x|\nabla X_{|J} (t)=0, \nabla
X_{|J'} (s)=0\bigr)\\
&&\hspace*{157pt}{}\times p_{\nabla X_{|J} (t), \nabla X_{|J'}
(s)}(0,0)
\\
&&\qquad{} \times\E\bigl\{ \bigl|\operatorname{det} \nabla^2 X_{|J}(t) \bigr|\bigl|
\operatorname{det} \nabla^2 X_{|J'}(s) \bigr| | X(t)=x, \nabla
X_{|J}(t)=0, \nabla X_{|J'}(s)=0 \bigr\}
\\
&&\quad\qquad:=I_3+I_4. %
\end{eqnarray*}
Note that
%
\begin{eqnarray}
\label{Eq:product} &&\bigl(J\times J'\bigr) \setminus\bigl(B\times
B'\bigr)
\nonumber
\\[-8pt]
\\[-8pt]
\nonumber
&&\qquad = \bigl( (J\setminus B) \times B' \bigr)
\cup \bigl( B\times\bigl(J'\setminus B'\bigr) \bigr)
\cup \bigl( (J\setminus B)\times\bigl(J'\setminus B'
\bigr) \bigr).
\end{eqnarray}
Since each product set on the right-hand side of (\ref{Eq:product})
consists of two sets with positive distance,
following the proof of Lemma~\ref{Lem:disjoint faces} we can verify
that $I_3$ is super-exponentially small.

For $I_4$, taking into account (\ref{Eq:super-exponentially small}),
one has
%
\begin{equation}
\label{Eq:neighboring faces special} %
\sup_{t\in B, s\in B'} \operatorname{Var} \bigl(X(t)|
\nabla X_{|J}(t), \nabla X_{|J'}(s) \bigr) \leq
\sigma_T^2- \ep_1. %
\end{equation}
For any $t\in B, s\in B'$ with $s \ne t$, in order to estimate
%
\begin{eqnarray}
\label{Eq:neighboring faces special 2}&& p_{\nabla X_{|J} (t),
 \nabla X_{|J'} (s)}(0,0)
 \nonumber
 \\[-8pt]
 \\[-8pt]
 \nonumber
 &&\qquad= (2\pi)^{-(k+k')/2} \bigl(\operatorname{ det Cov}
\bigl(\nabla X_{|J} (t), \nabla X_{|J'} (s)\bigr)
\bigr)^{-1/2},
\end{eqnarray}
we write the determinant on the right-hand side of (\ref{Eq:neighboring
faces special 2}) as
%
\begin{eqnarray}
\label{Eq:estimate covdet 1} %
&&\operatorname{ det Cov} \bigl(X_{m+1} (t), \ldots,
X_k (t), X_{k+1} (s), \ldots, X_{k+k'-1} (s)|\nonumber\\
&&\hspace*{82pt}X_{1} (t), \ldots, X_m (t), X_{1} (s),
\ldots, X_{m} (s)\bigr)
\\
&&\qquad{} \times\operatorname{ det Cov} \bigl(X_{1} (t), \ldots, X_m (t),
X_{1} (s), \ldots, X_{m} (s)\bigr) ,\nonumber %
\end{eqnarray}
where the first determinant in (\ref{Eq:estimate covdet 1}) is bounded
away from zero due to (H3). By~(H1),
as shown in \citet{Piterbarg96}, applying Taylor's formula, we
can write
%
\begin{equation}
\label{Formula:Taylor} \nabla X(s) = \nabla X(t) + \nabla^2 X(t)
(s-t)^T + \|s-t\|^{1+\eta} Y_{t,s},
\end{equation}
where $Y_{t,s}= (Y^1_{t,s}, \ldots, Y^N_{t,s})^T$ is a Gaussian vector
field with bounded variance uniformly for all
$t\in J$, $s\in J'$. Hence as $\|s-t\| \to0$, the second determinant
in~(\ref{Eq:estimate covdet 1}) becomes
%
\begin{eqnarray}
\label{Eq:estimate covdet 2} %
&& \operatorname{ det Cov} \bigl(X_{1} (t), \ldots,
X_m (t), X_{1} (t)+ \bigl\langle\nabla X_1(t), s-t
\bigr\rangle + \|s-t\|^{1+\eta}Y^1_{t,s} , \ldots,\nonumber
\\
&&\hspace*{137pt} X_m (t)+\bigl\langle\nabla X_m(t), s-t \bigr\rangle + \|s-t
\|^{1+\eta
}Y^m_{t,s}\bigr)\nonumber
\\
&&\qquad= \operatorname{ det Cov} \bigl(X_{1} (t), \ldots, X_m (t),\bigl \langle
\nabla X_1(t), s-t \bigr\rangle + \|s-t\|^{1+\eta}Y^1_{t,s}
, \ldots,
\nonumber
\\[-8pt]
\\[-8pt]
\nonumber
&&\hspace*{170pt} \bigl\langle\nabla X_m(t), s-t\bigr \rangle + \|s-t\|^{1+\eta
}Y^m_{t,s}
\bigr)
\\
&&\qquad= \|s-t\|^{2m}\operatorname{ det Cov} \bigl(X_{1} (t), \ldots,
X_m (t),\bigl \langle\nabla X_1(t), e_{t,s} \bigr\rangle ,
\ldots,\bigl \langle\nabla X_m(t), e_{t,s}\bigr\rangle \bigr)\nonumber\\
&&\qquad\quad{}\times\bigl(1+o(1)
\bigr),\nonumber %
\end{eqnarray}
where $e_{t,s}=(s-t)^T/\|s-t\|$ and due to (H3), the last
determinant in (\ref{Eq:estimate covdet 2}) is
bounded away from zero uniformly for all $t\in J$ and $s\in J'$. It
then follows from (\ref{Eq:estimate covdet 1})
and (\ref{Eq:estimate covdet 2}) that
%
\begin{equation}
\label{Eq:neighboring faces special 3} %
\operatorname{ det Cov} \bigl(\nabla X_{|J} (t),
\nabla X_{|J'} (s)\bigr) \geq C_1 \|s-t\|^{2m}
\end{equation}
for some constant $C_1 >0$. Similar to (\ref{Eq:disjoint faces 2}),
there exist constants $C_2, N_1>0$ such that
%
\begin{eqnarray}
\label{Eq:neighboring faces special 4} %
&&\sup_{t\in J, s\in J'}\E \bigl\{ \bigl|\operatorname{det}
\nabla^2 X_{|J}(t) \bigr|\bigl|\operatorname {det} \nabla^2
X_{|J'}(s) \bigr| |\nonumber\\
&&\hspace*{48pt} X(t)=x, \nabla X_{|J}(t)=0, \nabla
X_{|J'}(s)=0 \bigr\}
\\
&&\qquad \leq C_2\bigl(1+ x^{N_1}\bigr). \nonumber%
\end{eqnarray}
Combining (\ref{Eq:neighboring faces special}) with (\ref
{Eq:neighboring faces special 2}),
(\ref{Eq:neighboring faces special 3}) and (\ref{Eq:neighboring faces
special 4}), and noting
that $m< k'$ implies $1/\|s-t\|^m$ is integrable on $J\times J'$, we
conclude that $I_4$, and
hence $\E\{M_u (J) M_u (J') \}$ are finite and super-exponentially small.
\end{pf*}

\begin{theorem} \label{Thm:MEC approximation 2}
Let $X = \{X(t)\dvtx t\in\R^{N} \}$ be a centered Gaussian random field
with stationary
increments and $X(0) = 0$. Assume that \textup{(H1)}, \textup{(H2)} and
\textup{(H3)} are fulfilled.
Then there exists a constant $\alpha>0$ such that
%
\begin{eqnarray}
\label{Eq:MEC approximation 2} %
\P \Bigl\{\sup_{t\in T} X(t) \geq u
\Bigr\} &=& \sum_{k=0}^N\sum
_{J\in
\partial_k T} \E\bigl\{M_u^E (J) \bigr\} + o
\bigl(e^{-\alpha u^2 - u^2/{(2\sigma_T^2)} }\bigr)
\nonumber
\\[-8pt]
\\[-8pt]
\nonumber
&=& \E\bigl\{\varphi(A_u)\bigr\} +o\bigl(e^{-\alpha u^2 - u^2/{(2\sigma_T^2)} }\bigr),
\end{eqnarray}
where $\E\{\varphi(A_u)\}$ is the mean Euler characteristic of $A_u$
formulated in Theorem~\ref{Thm:MEC}.
\end{theorem}

The main idea for the proof of Theorem~\ref{Thm:MEC approximation 2}
comes from Aza\"is and Delmas
(\citeyear{AD02}) (especially their Theorem~4). Before showing the proof, we list
the following two lemmas whose
proofs are given in the \hyperref[app]{Appendix}.

\begin{lemma} \label{Lem:unifrom posdef} Under \textup{(H2)}, there
exists a constant $\alpha_0 >0$ such that
\[
\bigl\langle e, \bigl(\La-\La(t)\bigr)e\bigr\rangle \geq\alpha_0 \qquad\forall t\in T,
e\in \mathbb{S}^{N-1}.
\]
\end{lemma}

\begin{lemma} \label{Lem:rho}
Let $\{\xi_1(t)\dvtx t\in T_1\}$ and $\{\xi_2(t)\dvtx t\in T_2\}$ be
two centered
Gaussian random fields. Let
\begin{eqnarray*}
\sigma_i^2(t)& =& \operatorname{ Var}\bigl(
\xi_i(t)\bigr),\qquad \overline{\sigma}_i=\sup
_{t\in T_i} \sigma_i(t),\qquad \underline{
\sigma}_i=\inf_{t\in T_i} \sigma_i(t),
\\
\rho(t,s) &=& \frac{ \E\{\xi_1(t)\xi_2(s)\}}{\sigma_1(t)\sigma
_2(s)},\qquad \overline{\rho} =\sup_{t\in T_1, s\in T_2}
\rho(t,s),\qquad \underline{\rho}=\inf_{t\in
T_1, s\in T_2} \rho(t,s), %
\end{eqnarray*}
and assume $0<\underline{\sigma}_i\leq\overline{\sigma}_i<\infty$ for
$i=1,2$. If $0<\underline{\rho}
\leq\overline{\rho}<1$, then for any $N_1, N_2>0$, there exists a
constant $\alpha>0$ such that as $u \to\infty$,
\[
\sup_{t\in T_1, s\in T_2} \E\bigl\{\bigl(1+\bigl|\xi_1(t)\bigr|^{N_1}
+ \bigl|\xi_2(s)\bigr|^{N_2}\bigr) \mathbh{1}_{\{\xi_1(t) \geq u, \xi_2(s)<0\}} \bigr\}
= o\bigl(e^{-\alpha u^2 - u^2/(2\overline{\sigma}_1^2) }\bigr).
\]
Similarly, if $-1<\underline{\rho}\leq\overline{\rho}<0$, then
\[
\sup_{t\in T_1, s\in T_2} \E\bigl\{\bigl(1+\bigl|\xi_1(t)\bigr|^{N_1}
+ \bigl|\xi_2(s)\bigr|^{N_2}\bigr) \mathbh{1}_{\{\xi_1(t) \geq u, \xi_2(s)>0\}} \bigr\}
= o\bigl(e^{-\alpha u^2 - u^2/(2\overline{\sigma}_1^2) }\bigr).
\]
\end{lemma}

\begin{pf*}{Proof of Theorem~\ref{Thm:MEC approximation 2}} Note that
the second equality in (\ref{Eq:MEC approximation 2}) follows from
Theorem~\ref{Thm:MEC} and Lemma~\ref{Lem:mean M_u^E},
and similar to the proof in Theorem~\ref{Thm:MEC approximation 1}, we
only need to
show that $\E\{M_u^E (J) M_u^E (J') \}$ is super-exponentially small
when $J$ and $J'$
are neighboring. Let $I:= \bar{J}\cap\bar{J'} \neq\varnothing$. We
follow the assumptions in
(\ref{Eq:assumption for faces}) and assume also that all elements in
$\ep(J)$ and $\ep(J')$ are 1,
which implies $E(J)=\R^{N-k}_+$ and $E(J')=\R^{N-k'}_+$.

We first consider the case $k\geq1$. By the Kac--Rice metatheorem, $\E
\{M_u^E (J)\times  M_u^E (J') \}$ is bounded from above by
%
\begin{eqnarray}
\label{Eq:cross term} %
&&\int_{J} \,dt\int
_{J'} \,ds \int_u^\infty \,dx \int
_u^\infty \,dy \int_0^\infty
\,dz_{k+1} \cdots\nonumber\\
&&\quad{}\times\int_0^\infty
\,dz_{k+k'-m} \int_0^\infty
\,dw_{m+1} \cdots\int_0^\infty
\,dw_{k}\nonumber
\\
&&\quad{} \times\E\bigl\{ \bigl|\operatorname{det} \nabla^2 X_{|J}(t) \bigr|\bigl|\operatorname{det}
\nabla^2 X_{|J'}(s) \bigr| |\nonumber\\
&&\hspace*{38pt} X(t)=x, X(s)=y,\nabla
X_{|J}(t)=0, X_{k+1}(t)=z_{k+1},\ldots,
\nonumber
\\[-8pt]
\\[-8pt]
\nonumber
&& \hspace*{38pt} X_{k+k'-m}(t)=z_{k+k'-m}, \nabla X_{|J'}(s)=0,\\
&&\hspace*{130pt}X_{m+1}(s)=w_{m+1}, \ldots, X_{k}(s)=w_{k}
\bigr\}
\nonumber\\
&&\quad {}\times p_{t,s}(x,y,0, z_{k+1}, \ldots,z_{k+k'-m}, 0,
w_{m+1}, \ldots, w_{k} )
\nonumber\\
&&\qquad:= \int\!\!\int_{J\times J'} A(t,s) \,dt\,ds, \nonumber%
\end{eqnarray}
where $p_{t,s}(x,y,0, z_{k+1}, \ldots,z_{k+k'-m}, 0,w_{m+1}, \ldots,
w_{k} )$ is the density of
\[
\bigl(X(t),X(s),\nabla X_{|J}(t), X_{k+1}(t),\ldots,
X_{k+k'-m}(t), \nabla X_{|J'}(s), X_{m+1}(s), \ldots,
X_{k}(s)\bigr)
\]
evaluated at $(x,y,0, z_{k+1}, \ldots,z_{k+k'-m}, 0,w_{m+1}, \ldots,
w_{k} )$.

Let $\{e_1, e_2, \ldots, e_N\}$ be the standard orthonormal basis of
$\R
^N$. For $t\in J$ and $s\in J'$, let $e_{t, s}=(s-t)^T/\|s-t\|$ and let
$\alpha_i(t, s)= \langle e_i, (\La-\La(t))e_{t,s}\rangle $, then
%
\begin{equation}
\label{Eq:orthnormal basis decomp} \bigl(\La-\La(t)\bigr)e_{t,s}=\sum
_{i=1}^N \bigl\langle e_i, \bigl(\La-\La (t)
\bigr)e_{t,s}\bigr\rangle e_i = \sum
_{i=1}^N \alpha_i(t, s)
e_i.
\end{equation}
By Lemma~\ref{Lem:unifrom posdef}, there exists some $\alpha_0 >0$
such that
%
\begin{equation}
\label{Eq:posdef} \bigl\langle e_{t,s}, \bigl(\La-\La(t)\bigr)e_{t,s}
\bigr\rangle \geq\alpha_0
\end{equation}
for all $t$ and $s$. Under the assumptions (\ref{Eq:assumption for
faces}) and that all elements in $\ep(J)$ and $\ep(J')$ are 1, we have
the following representation:
\begin{eqnarray*}
t &=& (t_1, \ldots, t_m, t_{m+1},
\ldots, t_k, b_{k+1}, \ldots, b_{k+k'-m}, 0, \ldots,
0),
\\
s &=& (s_1, \ldots, s_m, b_{m+1}, \ldots,
b_k, s_{k+1}, \ldots, s_{k+k'-m}, 0, \ldots, 0),
\end{eqnarray*}
where $t_i \in(a_i, b_i)$ for all $i\in\sigma(J)$ and $s_j \in(a_j,
b_j)$ for all $j \in\sigma(J')$. Therefore,
%
\begin{eqnarray}
\label{Eq:e_k} %
\langle e_i , e_{t,s}\rangle &\geq&0\qquad
\forall m+1 \leq i \leq k,\nonumber
\\
\langle e_i , e_{t,s}\rangle &\leq&0\qquad \forall k+1\leq i \leq
k+k'-m,
\\
\langle e_i , e_{t,s}\rangle &=& 0 \qquad\forall k+k'-m<
i \leq N.\nonumber %
\end{eqnarray}
Let
%
\begin{eqnarray}\qquad
\label{Eq:D_k} %
D_i &= &\bigl\{ (t,s)\in J\times
J'\dvtx\alpha_i (t,s)\geq\beta_i \bigr\}\qquad
\mbox{if } m+1 \leq i \leq k,\nonumber
\\
D_i &=& \bigl\{ (t,s)\in J\times J'\dvtx
\alpha_i (t,s)\leq-\beta_i \bigr\}\qquad \mbox{if } k+1\leq i
\leq k+k'-m,
\\
D_0 &=& \Biggl\{ (t,s)\in J\times J'\dvtx\sum
_{i=1}^m \alpha_i (t,s)\langle
e_i , e_{t,s}\rangle \geq\beta_0 \Biggr\},\nonumber
\end{eqnarray}
where $\beta_0, \beta_1,\ldots, \beta_{k+k'-m}$ are positive constants
such that $\beta_0 + \sum_{i=m+1}^{k+k'-m} \beta_i < \alpha_0$.
It follows from (\ref{Eq:e_k}) and (\ref{Eq:D_k}) that, if $(t,s)$ does
not belong to any of $D_0, D_m, \ldots, D_{k+k'-m}$, then by (\ref
{Eq:orthnormal basis decomp}),
\[
\bigl\langle\bigl(\La-\La(t)\bigr)e_{t,s}, e_{t,s}\bigr \rangle = \sum
_{i=1}^N \alpha _i(t,s) \langle
e_i , e_{t,s}\rangle \leq\beta_0 + \sum
_{i=m+1}^{k+k'-m} \beta_i <
\alpha_0,
\]
which contradicts (\ref{Eq:posdef}). Thus, $D_0\,\cup
\,\bigcup_{i=m+1}^{k+k'-m} D_i$ is a covering of $J\times J'$, by (\ref
{Eq:cross term}),
\[
\E\bigl\{M_u^E (J) M_u^E
\bigl(J'\bigr) \bigr\} \leq\int\!\!\int_{D_0} A(t,s)
\,dt\,ds + \sum_{i=m+1}^{k+k'-m} \int\!\!\int
_{D_i} A(t,s) \,dt\,ds.
\]

We first show that $\int\!\!\int_{D_0} A(t,s) \,dt\,ds$ is
super-exponentially small. Similar to the proof of Theorem~\ref
{Thm:MEC approximation 1}, applying (\ref{Eq:neighboring faces special
2}), (\ref{Eq:neighboring faces special 3}) and (\ref{Eq:neighboring
faces special 4}), we obtain
%
\begin{eqnarray}
\label{Eq:integral on D0} %
&&\int\!\!\int_{D_0} A(t,s) \,dt\,ds\nonumber
\\
&&\qquad\leq\int\!\!\int_{D_0} \,dt\,ds \int_u^\infty
\,dx p_{\nabla X_{|J}(t),
\nabla X_{|J'} (s)}(0,0)\nonumber\\
&&\qquad\quad{}\times p_{X(t)}\bigl(x|\nabla X_{|J}(t)=0,
\nabla X_{|J'} (s)=0\bigr)\nonumber
\\
&&\qquad\quad{} \times\E\bigl\{\bigl |\operatorname{det} \nabla^2
X_{|J}(t)\bigr| \bigl|
\operatorname{det} \nabla ^2 X_{|J'}(s)\bigr| |\\
&&\hspace*{60pt} X(t)=x, \nabla
X_{|J}(t)=0, \nabla X_{|J'} (s)=0\bigr\}\nonumber
\\
&&\qquad\leq C'_1 \int\!\!\int_{D_0} \,dt\,ds
\int_u^\infty \,dx \bigl(1+x^{N_1}\bigr)
\|s-t\| ^{-m}\nonumber\\
&&\qquad\quad{}\times p_{X(t)}\bigl(x|\nabla X_{|J}(t)=0,
\nabla X_{|J'} (s)=0\bigr),\nonumber %
\end{eqnarray}
for some positive constants $C'_1$ and $N_1$. Due to Lemma~\ref
{Lem:disjoint faces}, we only need to consider the case when $\|s-t\|$
is small.
It follows from Taylor's formula (\ref{Formula:Taylor}) that as $\|
s-t\| \to0$,
%
\begin{eqnarray}
\label{Eq:conditional var on two gradians} %
&&\operatorname{Var} \bigl(X(t) | \nabla X_{|J} (t),
\nabla X_{|J'} (s)\bigr) \nonumber\\
&&\qquad\leq\operatorname {Var} \bigl(X(t) | X_1
(t),\ldots, X_m (t), X_1 (s),\ldots, X_m (s)
\bigr)\nonumber
\\
&&\qquad = \operatorname{Var} \bigl(X(t) | X_1 (t),\ldots, X_m (t),
X_{1} (t) + \bigl\langle\nabla X_1(t), s-t \bigr\rangle\nonumber\\
&&\hspace*{52pt}{} + \|s-t
\|^{1+\eta}Y^1_{t,s} , \ldots,\nonumber
\\
&& \hspace*{80pt}X_{m} (t) + \bigl\langle\nabla X_m(t), s-t \bigr\rangle + \|s-t\|
^{1+\eta}Y^m_{t,s}\bigr)
\nonumber
\\[-8pt]
\\[-8pt]
\nonumber
&&\qquad= \operatorname{Var} \bigl(X(t) | X_1 (t),\ldots, X_m (t), \bigl\langle
\nabla X_1(t), e_{t,s}\bigr\rangle + \|s-t\|^{\eta}Y^1_{t,s}
, \ldots,
\\
&& \hspace*{181pt}\bigl\langle\nabla X_m(t), e_{t,s}\bigr\rangle + \|s-t
\|^{\eta
}Y^m_{t,s}\bigr)\nonumber
\\
&&\qquad\leq\operatorname{Var} \bigl(X(t) |\bigl\langle\nabla X_1(t), e_{t,s}
\bigr\rangle + \|s-t\| ^{\eta
}Y^1_{t,s} , \ldots, \bigl\langle
\nabla X_m(t), e_{t,s}\bigr\rangle\nonumber \\
&&\hspace*{222pt}{}+ \|s-t\|^{\eta
}Y^m_{t,s}
\bigr)\nonumber
\\
&&\qquad= \operatorname{Var} \bigl(X(t) |\bigl\langle\nabla X_1(t), e_{t,s}\bigr\rangle
, \ldots, \bigl\langle \nabla X_m(t), e_{t,s}\bigr\rangle \bigr) + o(1).\nonumber
\end{eqnarray}
By Lemma~\ref{Lem:eigenvalue}, the eigenvalues of $[\operatorname{ Cov}(\langle
\nabla
X_1(t), e_{t,s}\rangle , \ldots, \langle\nabla X_m(t), e_{t,s}\rangle
)]^{-1}$ are
bounded uniformly in $t$ and $s$. Note that $\E\{X(t)\langle\nabla X_i(t),
e_{t,s}\rangle \}=-\alpha_i(t, s)$. Applying these facts to the last
line of
(\ref{Eq:conditional var on two gradians}), we see that there exist
constants $C_4>0$ and $\ep_0>0$ such that for $\|s-t\|$ sufficiently small,
%
\begin{eqnarray}
\label{Eq:integral on D0 2} %
\operatorname{Var} \bigl(X(t) | \nabla X_{|J} (t),
\nabla X_{|J'} (s)\bigr) &\leq&\sigma _T^2 -
C_4\sum_{i=1}^m
\alpha^2_i(t, s) + o(1)
\nonumber
\\[-8pt]
\\[-8pt]
\nonumber
&<& \sigma_T^2
- \ep_0, %
\end{eqnarray}
where the last inequality is due to the fact that $(t,s)\in D_0$ implies
\[
\sum_{i=1}^m
\alpha^2_i(t, s)\geq\sum_{i=1}^m
\alpha^2_i (t,s)\bigl|\langle e_i, e_{t,s}
\rangle\bigr |^2 \geq\frac{1}{m} \Biggl(\sum
_{i=1}^m \alpha_i (t,s)\langle
e_{t,s}, e_i \rangle \Biggr)^2 \geq
\frac{\beta_0^2}{m}. %
\]
Plugging (\ref{Eq:integral on D0 2}) into (\ref{Eq:integral on D0}) and
noting that $1/\|s-t\|^m$ is integrable on $J\times J'$, we conclude
that $\int\!\!\int_{D_0} A(t,s) \,dt\,ds$ is finite and super-exponentially small.

Next we show that $\int\!\!\int_{D_i} A(t,s) \,dt\,ds$ is super-exponentially
small for $i= m+1, \ldots, k$. It follows from (\ref{Eq:cross term})
that $\int\!\!\int_{D_i} A(t,s) \,dt\,ds$ is bounded by
%
\begin{eqnarray}
\label{Eq:integral on Di} %
&&\int\!\!\int_{D_i} \,dt\,ds \int
_u^\infty \,dx \int_0^\infty
\,dw_i p_{X(t),
\nabla X_{|J}(t), X_i(s), \nabla X_{|J'} (s)}(x,0,w_i,0)\nonumber
\\
&&\qquad\times\E\bigl\{ \bigl|\operatorname{det} \nabla^2 X_{|J}(t)\bigr| \bigl|
\operatorname{det} \nabla^2 X_{|J'}(s)\bigr| |\\
&&\hspace*{45pt}{} X(t)=x, \nabla
X_{|J}(t)=0 , X_i(s)=w_i, \nabla
X_{|J'} (s)=0 \bigr\}.\nonumber %
\end{eqnarray}
We can write
\begin{eqnarray*}
&&p_{X(t), X_i(s)}\bigl(x, w_i|X_i(t)=0\bigr)
\\
&&\qquad= \frac{1}{2\pi\sigma_1(t)\sigma
_2(t,s)(1-\rho^2(t,s))^{1/2}}
\\
&&\qquad\quad{} \times\exp \biggl\{-\frac{1}{2(1-\rho^2(t,s))} \biggl( \frac
{x^2}{\sigma_1^2(t)}+
\frac{w_i^2}{\sigma_2^2(t,s)} - \frac{2\rho(t,s)
xw_i}{\sigma_1(t)\sigma_2(t,s)} \biggr) \biggr\}, %
\end{eqnarray*}
where
\begin{eqnarray*}
\sigma_1^2(t) &=& \operatorname{Var} \bigl(X(t) |
X_i(t)=0\bigr),\qquad \rho(t,s) = \frac
{\E\{ X(t)X_i(s)| X_i(t)=0 \}}{\sigma_1(t)\sigma_2(t,s)},
\\
\sigma_2^2(t,s) & =&\operatorname{Var} \bigl(X_i(s) |
X_i (t)=0\bigr)= \frac{\operatorname{ det Cov}
(X_i(s), X_i(t))}{\lambda _{ii}}, %
\end{eqnarray*}
and $\rho^2(t,s) <1$ due to (H3). Therefore,
%
\begin{eqnarray}
\label{Eq:integral on Di 2} %
&&p_{X(t), \nabla X_{|J}(t), X_i(s), \nabla X_{|J'} (s)}(x,0,w_i,0)\nonumber
\\
&&\qquad = p_{\nabla X_{|J'} (s), X_1(t),\ldots,X_{i-1}(t),
X_{i+1}(t),\ldots, X_k(t)}\nonumber\\
&&\qquad\quad{}\times \bigl(0|X(t)=x, X_i(s)=w_i,X_i(t)=0
\bigr)
\nonumber
\\[-8pt]
\\[-8pt]
\nonumber
&& \qquad\quad{}\times p_{X(t), X_i(s)}\bigl(x, w_i|X_i(t)=0
\bigr)p_{X_i(t)}(0)
\\
&&\qquad \leq C_5\exp \biggl\{-\frac{1}{2(1-\rho^2(t,s))} \biggl(
\frac
{x^2}{\sigma_1^2(t)}+\frac{w_i^2}{\sigma_2^2(t,s)} - \frac{2\rho(t,s)
xw_i}{\sigma_1(t)\sigma_2(t,s)} \biggr) \biggr\}
\nonumber\\
&& \qquad\quad{}\times\bigl(\operatorname{detCov} \bigl(X(t), \nabla X_{|J}(t),
X_i(s), \nabla X_{|J'} (s)\bigr)\bigr)^{-1/2}\nonumber
\end{eqnarray}
for some positive constant $C_5$. Also, by an argument that is similar
to the proof of
Theorem~\ref{Thm:MEC approximation 1}, we can show that there exist
positive constants
$C_6, N_2$ and $N_3$ such that
%
\begin{eqnarray}
\label{Eq:integral on Di 3} \operatorname{ det Cov} \bigl(\nabla X_{|J} (t),
X_i(s), \nabla X_{|J'} (s)\bigr) &\geq& C_6^{-1}
\|s-t\|^{2(m+1)},
\\
\label{Ineq:sigma_2} C_6^{-1}\|s-t\|^2&\leq&
\sigma_2^2(t,s) \leq C_6\|s-t
\|^2
\end{eqnarray}
and
%
\begin{eqnarray}
\label{Eq:integral on Di 4} %
&&\E\bigl\{ \bigl|\operatorname{ det} \nabla^2
X_{|J} (t) \bigr| \bigl|\operatorname{ det} \nabla^2 X_{|J'}(s)\bigr| |\nonumber\\
&&\hspace*{14pt}X(t)=x, \nabla X_{|J}(t)=0 , X_i(s)=w_i,
\nabla X_{|J'} (s)=0 \bigr\}\nonumber
\\
&&\qquad=\E\bigl\{ \bigl|\operatorname{ det} \nabla^2 X_{|J} (t)\bigr |\bigl |\operatorname{ det}
\nabla^2 X_{|J'}(s)\bigr| | X(t)=x, \nabla X_{|J}(t)=0
,
\\
&& \hspace*{64pt}\bigl\langle\nabla X_i(t), e_{t,s}\bigr\rangle =w_i/\|s-t\|
+o(1), \nabla X_{|J'} (s)=0 \bigr\}\nonumber
\\
&&\qquad\leq C_7\bigl(x^{N_2}+\bigl(w_i/\|s-t
\|\bigr)^{N_3} +1 \bigr).\nonumber %
\end{eqnarray}
Combining (\ref{Eq:integral on Di}) with (\ref{Eq:integral on Di 2}),
(\ref{Eq:integral on Di 3}) and (\ref{Eq:integral on Di 4}), and making
change of variable $w= w_i/\|s-t\|$, we obtain that for some positive
constant $C_8$,
%
\begin{eqnarray}
\label{Eq:integral on the Di} %
&&\int\!\!\int_{D_i} A(t,s) \,dt\,ds\nonumber
\\
&&\qquad\leq C_8 \int\!\!\int_{D_i} \,dt\,ds \|s-t
\|^{-m-1} \int_u^\infty \,dx \int
_0^\infty \,dw_i\nonumber\\
&&\qquad\quad{}\times \bigl(x^{N_2}+\bigl(w_i/
\|s-t\|\bigr)^{N_3} +1 \bigr)
\nonumber
\\[-8pt]
\\[-8pt]
\nonumber
&&\qquad\quad{} \times\exp \biggl\{-\frac{1}{2(1-\rho^2(t,s))} \biggl( \frac
{x^2}{\sigma_1^2(t)}+
\frac{w_i^2}{\sigma_2^2(t,s)} - \frac{2\rho(t,s)
xw_i}{\sigma_1(t)\sigma_2(t,s)} \biggr) \biggr\}
\\
&&\qquad= C_8\int\!\!\int_{D_i} \,dt\,ds \|s-t
\|^{-m} \int_u^\infty \,dx \int
_0^\infty \,dw \bigl(x^{N_2}+w^{N_3}
+1 \bigr)\nonumber
\\
&&\qquad\quad{} \times\exp \biggl\{-\frac{1}{2(1-\rho^2(t,s))} \biggl( \frac
{x^2}{\sigma_1^2(t)}+
\frac{w^2}{\wt{\sigma}_2^2(t,s)} - \frac
{2\rho
(t,s) xw}{\sigma_1(t)\wt{\sigma}_2(t,s)} \biggr) \biggr\},\nonumber %
\end{eqnarray}
where $\wt{\sigma}_2(t,s) = \sigma_2(t,s)/\|s-t\|$ is bounded by
(\ref
{Ineq:sigma_2}). Applying Taylor's formula~(\ref{Formula:Taylor}) to
$X_i(s)$ and noting that $\E\{X(t)\langle\nabla X_i(t), e_{t,s}\rangle \}
=-\alpha
_i(t, s)$, we obtain
\begin{eqnarray*}
\rho(t,s) &=&\frac{1}{\sigma_1(t)\sigma_2(t,s)} \biggl(\E\bigl\{ X(t)X_i(s)
\bigr\} - \frac{1}{\lambda _{ii}}\E\bigl\{ X(t)X_i(t)\bigr\}\E\bigl\{
X_i(s)X_i(t)\bigr\} \biggr)
\\
&=& \frac{\|s-t\|}{\sigma_1(t)\sigma_2(t,s)} \biggl( -\alpha_i(t, s)+\| s-t
\|^{\eta}\E\bigl\{ X(t)Y^i_{t,s}\bigr\}
\\
&&\hspace*{62pt}{} - \frac{\|s-t\|^{\eta}}{\lambda _{ii}}\E\bigl\{ X(t)X_i(t)\bigr\}\E\bigl\{
X_i(t)Y^i_{t,s}\bigr\} \biggr). %
\end{eqnarray*}
By (\ref{Ineq:sigma_2}) and the fact that $(t, s)\in D_i$ implies
$\alpha_i (t,s)\geq\beta_i >0$ for $i= m+1, \ldots, k$, we conclude
that $\rho(t,s)\leq-\delta_0$ for some $\delta_0>0$ uniformly for
$t\in J$, $s\in J'$ with $\|s-t\|$ sufficiently small. Then
applying Lemma~\ref{Lem:rho} to (\ref{Eq:integral on the Di}) yields
that $\int\!\!\int_{D_i} A(t,s) \,dt\,ds$ is super-exponentially small.

It is similar to prove that $\int\!\!\int_{D_i} A(t,s) \,dt\,ds$ is
super-exponentially small for $i=k+1, \ldots, k+k'-m$. In fact, in such
case, $\int\!\!\int_{D_i} A(t,s) \,dt\,ds$ is bounded by
\begin{eqnarray*}
&&\int\!\!\int_{D_i} \,dt\,ds \int_u^\infty
\,dx \int_0^\infty \,dz_i
p_{X(t),
\nabla X_{|J}(t), X_i(t), \nabla X_{|J'} (s)}(x,0,z_i,0)
\\
&&\qquad{} \times\E\bigl\{ \bigl|\operatorname{det} \nabla^2 X_{|J}(t)\bigr| \bigl|
\operatorname{det} \nabla^2 X_{|J'}(s)\bigr| |\\
&&\hspace*{48pt}{} X(t)=x, \nabla
X_{|J}(t)=0 , X_i(t)=z_i, \nabla
X_{|J'} (s)=0 \bigr\}. %
\end{eqnarray*}
We can follow the proof in the previous stage by exchanging the
positions of $X_i(s)$ and $X_i(t)$ and replacing $w_i$ with $z_i$. The
details are omitted since the procedure is very similar.

If $k=0$, then $m=0$ and $\sigma(J')=\{1, \ldots, k'\}$. Since $J$
becomes a singleton, we may let $J=\{t_0\}$. By the Kac--Rice
metatheorem, $\E\{M_u (J) M_u (J') \}$ is bounded by
\begin{eqnarray*}
&&\int_{J'} \,ds \int_u^\infty
\,dx \int_u^\infty \,dy \int_0^\infty
\,dz_{1} \cdots\int_0^\infty
\,dz_{k'} p_{t_0,s}(x,y,z_{1}, \ldots,z_{k'},
0)
\\
&&\quad {}\times\E\bigl\{ \bigl|\operatorname{det} \nabla^2 X_{|J'}(s) \bigr| |\\
&&\hspace*{36pt}{}X(t_0)=x, X(s)=y, X_{1}(t_0)=z_{1},
\ldots, X_{k'}(t_0)=z_{k'}, \nabla
X_{|J'}(s)=0 \bigr\}
\\
&&\qquad:= \int_{J'} \wt{A}(t_0,s) \,ds, %
\end{eqnarray*}
where $p_{t_0,s}(x,y,z_{1}, \ldots,z_{k'}, 0)$ is the density of
$(X(t_0),X(s),X_{1}(t_0),\ldots, X_{k'}(t_0),\break  \nabla X_{|J'}(s))$
evaluated at $(x,y,z_{1}, \ldots,z_{k'}, 0)$. Similarly, $J'$ could be
covered by $\bigcup_{i=1}^{k'}\wt{D}_i$ with $\wt{D}_i = \{ s\in
J'\dvtx\alpha
_i (t_0,s)\leq-\wt{\beta}_i \}$ for some positive constants $\wt
{\beta
}_i$, $1\leq i \leq k'$. On the other hand,
\begin{eqnarray*}
\int_{\wt{D}_i} \wt{A}(t_0,s) \,ds &\leq&
\int_{\wt{D}_i} \,ds \int_u^\infty \,dx
\int_0^\infty \,dz_{i}
p_{X(t_0), X_{i}(t_0), \nabla
X_{|J'}(s)}(x,z_{i},0)
\\
& &{}\times\E\bigl\{ \bigl|\operatorname{det} \nabla^2 X_{|J'}(s) \bigr| |
X(t_0)=x, X_{i}(t_0)=z_{i},
\nabla X_{|J'}(s)=0 \bigr\}. %
\end{eqnarray*}
By similar discussions, we obtain that $\E\{M_u^E (J) M_u^E (J') \}$
is super-exponen\-tially small, and hence complete the proof.
\end{pf*}

\section{Further remarks and examples}\label{sec5}

\begin{remark}[(The case when $T$
contains the origin)]\label{Remark:contain the origin}
We now show that the conclusions of Theorems \ref{Thm:MEC
approximation 1}
and \ref{Thm:MEC approximation 2} still hold when $T$ contains
the origin.
In such case, condition (H3) is actually not satisfied since
$X(0)=0$ is degenerate
[this is in contrast with the case when $X=\{X(t), t \in\R^N\}$ is
assumed to have constant variance].
We will in fact prove a more general result. Let $T_0 \subset T$ be a
finite union of compact
rectangles such that $\sup_{t\in T_0} \nu(t) < \sigma_T^2$, then
according to the Borell--TIS inequality, $\P\{\sup_{t\in T_0} X(t)\geq
u\}$
is super-exponentially small. Let $\widehat{T}=T\setminus T_0$, then
%
\begin{eqnarray}
\label{Eq:containing origin} \P \Bigl\{\sup_{t\in\widehat{T}} X(t)\geq u \Bigr\}
&\leq&\P
\Bigl\{ \sup_{t\in T} X(t)\geq u \Bigr\}
\nonumber
\\[-8pt]
\\[-8pt]
\nonumber
&\leq&\P \Bigl\{\sup
_{t\in\widehat{T}} X(t)\geq u \Bigr\} + \P \Bigl\{ \sup
_{t\in T_0} X(t)\geq u \Bigr\}.
\end{eqnarray}
To estimate $\P\{\sup_{t\in\widehat{T} } X(t)\geq u\}$, similar to
the rectangle $T$, we decompose
$\widehat{T}$ into several faces by lower dimensions such that
$\widehat{T}
= \bigcup_{k=0}^N \partial_k \widehat{T}= \bigcup_{k=0}^N \bigcup_{L\in\partial
_k \widehat{T}} L$.
Then we can get the bounds similar to (\ref{Ineq:lowerbound and upperbound})
with $T$ replaced with $\widehat{T}$ and $J$ replaced with $L$.
Following the proof of Theorem~\ref{Thm:MEC approximation 1} yields
\[
\P \Bigl\{\sup_{t\in\widehat{T}} X(t) \geq u \Bigr\} = \sum
_{k=0}^N\sum_{L\in\partial_k \widehat{T}} \E
\bigl\{M_u (L) \bigr\} + o\bigl(e^{-\alpha u^2 - u^2/{(2\sigma_T^2)} }\bigr).
\]
Because $\sup_{t\in T_0} \nu(t) < \sigma_T^2$, 
we can show that terms $\E\{M_u (L) \}$ are
super-exponentially small for all faces $L$ such that $L \subset
\partial_k \bar{T}_0$ with
$0\leq k\leq N-1$. The same reasoning yields that
for $1\leq k\leq N$, $L\in\partial_k
\widehat{T}$, $J\in\partial_k T$ such that $L\subset J$, the
difference between
$\E\{M_u (L) \}$ and $\E\{M_u (J) \}$ is super-exponentially small.
Hence, we obtain
%
\begin{eqnarray}
\label{Eq:containing origin 1} %
&&\P \Bigl\{\sup_{t\in\widehat{T}} X(t) \geq u
\Bigr\}\nonumber\\
&&\qquad= \sum_{\{t\}
\in
\partial_0 T} \Psi \biggl(\frac{u}{\sqrt{\nu(t)}}
\biggr) + \sum_{k=1}^N\sum
_{J\in
\partial
_k T} \frac{1}{(2\pi)^{(k+1)/2}|\La_J|^{1/2}}
\\
&&\qquad\quad{} \times\int_J \frac{|\La_J-\La_J(t)|}{\theta^k_t} H_{k-1}
\biggl(\frac{u}{\theta_t} \biggr)e^{-u^2/{(2\theta^2_t)}} \,dt + o\bigl(e^{-\alpha u^2 - u^2/{(2\sigma_T^2) }}
\bigr).\nonumber %
\end{eqnarray}
Here, by convention, if $\theta_t =0$, we regard $e^{-u^2/{(2\theta
^2_t)}}$ as 0. Combining
(\ref{Eq:containing origin}) with~(\ref{Eq:containing origin 1}), we
conclude that Theorem~\ref{Thm:MEC approximation 1} still holds when $T$ contains the origin.
The argument for
Theorem~\ref{Thm:MEC approximation 2} is similar.
\end{remark}

\begin{remark}
Based on the proofs of Theorems \ref{Thm:MEC approximation 1} and \ref
{Thm:MEC approximation 2},
one may expect that the approximation (\ref{Eq:MEC approx error}) holds
for a much wider class
of smooth Gaussian fields (not necessarily with stationary increments).
Meanwhile, the argument
for the parameter set could go far beyond the rectangle case. These
further developments are in \citet{ChengPhD}.
\end{remark}

\begin{remark}[(Refinements of Theorem~\ref{Thm:MEC approximation 1})]\label{Example:unique maximum 1}
Let $X$ be a Gaussian field as in Theorem~\ref{Thm:MEC approximation
1}. Suppose that $\nu(t_0)= \sigma_T^2$
for some $t_0 \in J\in\partial_k T$ ($k \ge0$) and $\nu(t)< \sigma
_T^2$ for all $t \in T\setminus\{ t_0\}$.

(i) If $k=0$, then, due to (\ref{Condition:boundary max point}),
$\sup_{t\in T\setminus\{t_0\}}
\theta_t^2\leq\sigma_T^2-\ep_0$ for some $\ep_0>0$. This implies that
$\E\{M_u(J')\}$ are super-exponentially small for all faces $J'$ other
than $\{t_0\}$. Therefore, there
is a constant $\alpha> 0$ such that
%
\begin{equation}
\label{Eq:maximum at corner} \P \Bigl\{\sup_{t\in T} X(t) \geq u \Bigr\} = \Psi
\biggl(\frac
{u}{\sigma
_T} \biggr)+ o \bigl(e^{ - \alpha u^2 -u^2/(2\sigma_T^2) } \bigr)\qquad \mbox{as } u
\to\infty.
\end{equation}
For example, let $Y$ be a stationary isotropic Gaussian field with
covariance $\rho(t)=e^{-\|t\|^2}$ and
define $X(t)=Y(t)-Y(0)$. Then $X$ is a smooth Gaussian field with
stationary increments satisfying
conditions (H1)--(H3). Let $T=[0,1]^N$, then we can apply
(\ref{Eq:maximum at corner})
to approximate the excursion probability of $X$ with $t_0=(1, \ldots, 1)$.

(ii) If $k \geq1$, then similarly, $\E\{M_u(J')\}$ are
super-exponentially small
for all faces $J'\neq J$. It follows from Theorem~\ref{Thm:MEC
approximation 1} that
\[
\P \Bigl\{\sup_{t\in T} X(t) \geq u \Bigr\} =
\frac{u^{k-1}}{(2\pi
)^{(k+1)/2}|\La_J|^{1/2}} \int_J \frac{|\La_J-\La_J(t)|}{\theta^{2k-1}_t}
e^{-u^2/{(2\theta
^2_t)}} \,dt\bigl(1+o(1)\bigr). %
\]
Let $\tau(t)= \theta^2_t$, then $\forall i \in\sigma(J)$, $\tau
_i(t_0)=0$, since $t_0$ is a local
maximum point of $\tau$ restricted on $J$. Assume additionally that the
Hessian matrix
%
\begin{equation}
\label{Def:Theta} \Theta_J(t_0):= \bigl(
\tau_{ij}(t_0)\bigr)_{i, j \in\sigma(J)}
\end{equation}
(here $\tau_{ij} = \partial^2 \tau/\partial t_i\, \partial t_j$) is
negative definite,
then the Hessian matrix of $1/{(2\theta^2_t)}$ at $t_0$ restricted on $J$,
\[
\wt{\Theta}_{J}(t_0)= -\frac{1}{2\tau^2(t_0)} \bigl(
\tau_{ij}(t_0)\bigr)_{i, j
\in\sigma(J)} = -\frac{1}{2\sigma^4_T}
\Theta_J(t_0),
\]
is positive definite. Let $g(t) = |\La_J-\La_J(t)|/\theta^{2k-1}_t$ and
$h(t) = 1/(2\theta^2_t)$,
applying Lemma~\ref{Lem:Laplace method 1} in the \hyperref
[app]{Appendix} with $T$
replaced by $J$ yields that as $u \to\infty$,
%
\begin{eqnarray}
\label{Eq:asymt for unique maximum 1} %
&&\P \Bigl\{\sup_{t\in T} X(t) \geq u
\Bigr\} \nonumber\\
&&\qquad= \frac{u^{k-1}|\La
_J-\La
_J(t_0)|}{(2\pi)^{(k+1)/2}|\La_J|^{1/2}
\theta^{2k-1}_{t_0}} \frac{(2\pi)^{k/2}}{u^{k} |\wt{\Theta
}_{J}(t_0)|^{1/2}} e^{-u^2/(2\theta^2_{t_0})} \bigl(1+o(1)
\bigr)
\\
&&\qquad = \frac{ 2^{k/2} |\La_J-\La_J(t_0)|}{|\La_J|^{1/2}|-\Theta
_{J}(t_0)|^{1/2}}\Psi \biggl(\frac{u}{\sigma_T} \biggr) \bigl(1+o(1)\bigr).\nonumber
\end{eqnarray}
\end{remark}

\begin{example}[(The cosine field)] \label{Example:unique maximum 1b} We
consider the \emph{cosine random field} $Z$ on $\R^2$ [cf. Adler and
Taylor (\citeyear{AT07}), page~382]:
\[
Z(t)=\frac{1}{\sqrt{2}}\sum_{i=1}^2
\bigl(\xi_i\cos t_i + \xi_i'\sin
t_i\bigr),\qquad t=(t_1,t_2)\in\R^2,
\]
where $\xi_1$, $\xi_1'$, $\xi_2$, $\xi_2'$ are independent, standard
Gaussian variables. Clearly
$Z$ is a centered, unit-variance and smooth stationary Gaussian field.
Moreover, $Z$ is periodic
and $Z(t)=-Z_{11}(t)-Z_{22}(t)$. Let $X(t)=Z(t)-Z(0)$ 
for all $t\in T\subset[0,2\pi)^2$. 
Then $X$ is a centered and smooth Gaussian field with stationary
increments with $X(0) = 0$.
The variogram and covariance of $X$ are given respectively by
%
\begin{eqnarray}
\label{Eq:cov in cos field} %
\nu(t)&=& 2-\cos t_1 - \cos t_2,
\nonumber
\\[-8pt]
\\[-8pt]
\nonumber
C(t,s)&=&1+ \frac{1}{2}\sum_{i=1}^2
\bigl[\cos(t_i-s_i) - \cos t_i - \cos
s_i\bigr]. %
\end{eqnarray}
Taking the partial derivatives of $C$ gives
%
\begin{eqnarray}
\label{Eq:derivatives in cos field} %
\E\bigl\{X(t)\nabla X(t)\bigr\} &=& \tfrac{1}{2}(
\sin t_1, \sin t_2)^T,\qquad \La = \operatorname{ Cov}\bigl(
\nabla X(t)\bigr) =\tfrac{1}{2} I_2,
\nonumber
\\[-8pt]
\\[-8pt]
\nonumber
\La- \La(t)&=&-\E\bigl\{X(t)\nabla^2 X(t)\bigr\}= \tfrac{1}{2}
\bigl[I_2 - \operatorname{ diag}(\cos t_1, \cos t_2)
\bigr], %
\end{eqnarray}
where $I_2$ is the $2\times2$ identity matrix and diag denotes the
diagonal matrix. Therefore, $X$
satisfies conditions (H1) and (H2) on $ T\subset
(0,2\pi)^2
$. Even though
condition (H3) is not fully satisfied [because
$X_{12}(t)\equiv0$
and $X(t)=-2X_{11}(t)-2X_{22}(t)$ when $t=(\pi,\pi)$], it can be shown
that this does not affect the validity of Theorems~\ref{Thm:MEC}, \ref
{Thm:MEC approximation 1} and
\ref{Thm:MEC approximation 2} for the random field $\{Z(t) - Z(0), t
\in\R^2\}$ with $ T\subset(0,2\pi)^2$.

(i) Let $T=[0, \pi/2]^2$. Then by (\ref{Eq:cov in cos field}), $\nu
(t)$ attains its maximum 2 only at the upper-right
vertex $(\pi/2, \pi/2)$, where both partial derivatives of $\nu$ are
positive. By Remark~\ref{Remark:contain the origin}
(with $T_0 = [0, \varepsilon]\times[0, \pi/2] \cup[0, \pi/2]\times[0,
\varepsilon]$, where $\varepsilon>0$ is sufficiently small) and the
result (i) in Remark~\ref{Example:unique maximum 1}, we obtain $\P\{\sup_{t\in T} X(t)
\geq
u \} = \Psi(u/\sqrt{2})( 1+ o(e^{-\alpha u^2}))$.

(ii) Let $T=[0, 3\pi/2]\times[0, \pi/2]$. Then $\nu(t)$ attains its
maximum 3 only at the boundary point
$t^*=(\pi, \pi/2)$, where $\nu_2(t^*)>0$ so that condition (\ref
{Condition:boundary max point}) is satisfied.
In this case, $t^*\in J=(0, 3\pi/2)\times\{\pi/2\}$. By (\ref
{Eq:derivatives in cos field}), we obtain $\La_J =
\frac{1}{2}$ and $\La_J-\La_J(t^*)=\frac{1}{2}(1-\cos t_1^*)=1$. On the
other hand, for $t\in J$, by
(\ref{Eq:derivatives in cos field}),
%
\begin{equation}
\tau(t)=\theta^2_t=\operatorname{ Var}\bigl(X(t)|X_1(t)
\bigr)= 2-\cos t_1 - \cos t_2 - \tfrac
{1}{2}
\sin^2 t_1.
\end{equation}
Therefore, $\Theta_{J}(t^*)=\tau_{11}(t^*)= -2$. By plugging these into
(\ref{Eq:asymt for unique maximum 1}) with $k=1$,
we have $\P\{\sup_{t\in T} X(t) \geq u \} = \sqrt{2}\Psi(u/\sqrt{3})(
1+ o(1))$.

(iii) Let $T=[0, 3\pi/2]^2$. Then $\nu(t)$ attains its maximum 4
only at the interior point $t^*=(\pi, \pi)$.
In this case, $t^*\in J=(0, 3\pi/2)^2$. By (\ref{Eq:derivatives in cos
field}), we obtain $\La_J = \frac{1}{2}I_2$
and $\La_J-\La_J(t^*)=I_2$. On the other hand, for $t\in J$, by (\ref
{Eq:derivatives in cos field}),
%
\begin{eqnarray}
\tau(t)=\theta^2_t&=&\operatorname{ Var}\bigl(X(t)|X_1(t),X_2(t)
\bigr)
\nonumber
\\[-8pt]
\\[-8pt]
\nonumber
&=& 2-\cos t_1 - \cos t_2 - \tfrac{1}{2}
\sin^2 t_1 - \tfrac{1}{2}\sin^2
t_2.
\end{eqnarray}
Therefore, $\Theta_{J}(t^*)=(\tau_{ij}(t^*))_{i,j=1,2}= -2I_2$. By
plugging these into (\ref{Eq:asymt for unique maximum 1})
with $k=2$, we obtain $\P\{\sup_{t\in T} X(t) \geq u \} = 2\Psi(u/2)(
1+ o(1))$.
\end{example}

\begin{remark}[(Refinements of Theorem~\ref{Thm:MEC approximation 2})]\label{Example:unique maximum 2}
Let $X$ be a Gaussian field as in Theorem~\ref{Thm:MEC approximation
2}. Suppose $t_0\in J\in\partial_k T$
is the only point in $T$ such that $\nu(t_0)= \sigma_T^2$. Assume
$\sigma(J)=\{1,\ldots,k\}$, all elements
in $\ep(J)$ are 1, $\nu_{k'}(t_0) =0$ for all $k+1\leq k' \leq N$. Then
by Theorem~\ref{Thm:MEC approximation 2},
%
\begin{eqnarray}
\label{Eq:unique max point 1}&&\P \Bigl\{\sup_{t\in T} X(t) \geq u \Bigr\}
\nonumber
\\[-8pt]
\\[-8pt]
\nonumber
&&\qquad = \E
\bigl\{M_u^E (J) \bigr\} + \sum
_{k'=k+1}^N \sum_{J'\in\partial_{k'} T, \bar{J'}\cap\bar{J}\neq\varnothing} \E
\bigl\{ M_u^E \bigl(J'\bigr) \bigr\} + o
\bigl(e^{-\alpha u^2 - u^2/{(2\sigma_T^2)} }\bigr).\hspace*{-20pt}
\end{eqnarray}
Lemma~\ref{Lem:mean M_u^E} indicates $\E\{M_u^E(J)\}=(-1)^k \E\{\sum^k_{i=0} (-1)^i \mu_i(J)\}(1+ o(e^{-\alpha x^2}))$.
Therefore,
%
\begin{eqnarray}
\label{Eq:reduced ME} %
\E\bigl\{M_u^E(J)\bigr\} &=&
(-1)^k\int_J p_{\nabla X_{|J}(t)}(0) \,dt\nonumber\\
&&{}\times \E\bigl\{
\operatorname{ det} \nabla^2 X_{|J}(t) \mathbh{1}_{\{(X_{k+1}(t),\ldots, X_N(t))\in
\R
^{N-k}_+ \}}
\nonumber\\
&&\hspace*{82pt}{} \times\mathbh{1}_{\{X(t)\geq u\}
} |\nabla X_{|J}(t)=0 \bigr\}
\bigl(1+ o\bigl(e^{-\alpha x^2}\bigr)\bigr)\nonumber
\\
&=& \int_u^\infty \,dx \int_J
\,dt \frac{(-1)^k e^{-x^2/{(2\theta
^2_t)}}}{(2\pi)^{(k+1)/2}|\La_J|^{1/2}\theta_t} \\
&&{}\times \E\bigl\{ \operatorname{ det} \nabla^2
X_{|J}(t)\mathbh{1}_{\{(X_{k+1}(t),\ldots,X_N(t))\in\R^{N-k}_+ \}}|
\nonumber\\
&&\hspace*{92pt} X(t)=x, \nabla X_{|J}(t)=0 \bigr\} \bigl(1+ o\bigl(e^{-\alpha u^2}
\bigr)\bigr)
\nonumber\\
&:=& \int_u^\infty A_J(x) \,dx\bigl(1+
o\bigl(e^{-\alpha u^2}\bigr)\bigr).\nonumber %
\end{eqnarray}
Similarly, we have
\begin{eqnarray*}
&&\E\bigl\{M_u^E\bigl(J'\bigr)
\bigr\}\nonumber\\
&&\qquad= \int_u^\infty \,dx \int
_{J'} \,dt \frac{(-1)^{k'}
e^{-x^2/{(2\theta^2_t)}}}{(2\pi)^{(k'+1)/2}|\La_{J'}|^{1/2}\theta
_t}
\nonumber
\\[-8pt]
\\[-8pt]
\nonumber
&&\qquad\quad{}\times\E \bigl\{ \operatorname{ det}
\nabla^2 X_{|J'}(t)
\mathbh{1}_{\{(X_{J'_1}(t),\ldots,X_{J'_{N-k'}}(t))\in\R
^{N-k'}_+ \}} |\\
&&\hspace*{136pt} X(t)=x, \nabla X_{|J'}(t)=0 \bigr
\} \bigl(1+ o\bigl(e^{-\alpha u^2}\bigr)\bigr).\nonumber %
\end{eqnarray*}

(i) First, we consider the case $k \geq1$. We use the same notation
$\tau(t)$,
$\Theta_J(t)$ and $\wt{\Theta}_J(t)$ in Remark~\ref{Example:unique
maximum 1}. Let $h(t) = 1/(2\theta^2_t)$ and
\begin{eqnarray*}
&&g_x(t)= \frac{(-1)^k}{(2\pi)^{(k+1)/2}|\La_J|^{1/2}\theta_t}\E\bigl\{ \operatorname{ det}
\nabla^2 X_{|J}(t) \mathbh{1}_{\{(X_{k+1}(t),\ldots,X_N(t))\in\R^{N-k}_+ \}}
|\\
&&\hspace*{209pt}X(t)=x, \nabla X_{|J}(t)=0 \bigr\}. %
\end{eqnarray*}
Note that $\sup_{t\in T} |g_x(t)| = o(x^{N_1})$ for some $N_1>0$ as
$x\to\infty$, which implies that
the growth of $g_x(t)$ can be dominated by the exponential decay
$e^{-x^2h(t)}$, hence both Lemmas
\ref{Lem:Laplace method 1} and \ref{Lem:Laplace method 2} in the
\hyperref[app]{Appendix} are still applicable. Applying Lemma~\ref{Lem:Laplace method 1} with $T$ replaced by $J$ and $u$ replaced by
$x^2$, we obtain that as $x\to\infty$,
%
\begin{equation}
\label{Eq:A_Jx} %
A_J(x) =\frac{(2\pi)^{k/2} }{x^k (\operatorname{ det} \wt{\Theta
}_J(t_0))^{1/2}}g_x(t_0)
e^{-x^2/(2\sigma_T^2)} \bigl(1+ o(1)\bigr). %
\end{equation}
On the other hand, it follows from (\ref{Eq:formula for det 2}) that
\begin{eqnarray*}
g_x(t)&=& \frac{1}{(2\pi)^{(k+1)/2}|\La_J|^{1/2}\theta_t}\int\cdots \int
_{\R^{N-k}_+} \,dy_{k+1} \cdots \,dy_N
\\
&&{} \times\frac{|\La_{J}-\La_{J}(t)|}{\gamma_t^k} H_k \biggl(\frac
{x}{\gamma_t}+
\gamma_tC_{k+1}(t)y_{k+1} + \cdots+
\gamma_tC_N(t)y_N \biggr)
\\
&&{} \times p_{X_{k+1}(t),\ldots,X_N(t)}\bigl(y_{k+1}, \ldots, y_N|
X(t)=x, \nabla X_{|J}(t)=0\bigr). %
\end{eqnarray*}
Note that, under the assumptions on $X$ at the beginning of Remark~\ref
{Example:unique maximum 2},
$X(t_0)$ and $\nabla X(t_0)$ are independent, and $C_j(t_0)=0$
for all $1\leq j \leq N$. Therefore,
\begin{eqnarray*}
g_x(t_0)&=& \frac{|\La_{J}-\La_{J}(t_0)|}{(2\pi)^{(k+1)/2}|\La
_J|^{1/2}\sigma_T^{k+1}}
H_k \biggl(\frac{x}{\sigma_T} \biggr)
\\
&&{} \times\P\bigl\{\bigl(X_{k+1}(t_0),\ldots,X_N(t_0)
\bigr)\in\R^{N-k}_+|\nabla X_{|J}(t_0)=0\bigr\}.
\end{eqnarray*}
Plugging this and (\ref{Eq:A_Jx}) into (\ref{Eq:reduced ME}), we obtain
%
\begin{eqnarray}
\label{Eq:unique max point 4} %
&&\E\bigl\{M_u^E(J)\bigr\}\nonumber\\
&&\qquad =
\frac{ 2^{k/2} |\La_{J}-\La_{J}(t_0)|}{|\La
_{J}|^{1/2}|-\Theta_{J}(t_0)|^{1/2}} \Psi \biggl(\frac{u}{\sigma_T} \biggr)
\\
&&\qquad\quad{} \times\P\bigl\{\bigl(X_{k+1}(t_0),\ldots,X_N(t_0)
\bigr)\in\R^{N-k}_+|\nabla X_{|J}(t_0)=0\bigr\}
\bigl(1+o(1)\bigr).\nonumber %
\end{eqnarray}
For $J'\in\partial_{k'} T$ with $\bar{J'}\cap\bar{J}\neq
\varnothing$, we
apply Lemma~\ref{Lem:Laplace method 2} with $T$ replaced by $J'$
to obtain
%
\begin{eqnarray}
\label{Eq:unique max point 5} %
&&\E\bigl\{M_u^E
\bigl(J'\bigr)\bigr\}\nonumber \\
&&\qquad = \frac{ 2^{k'/2} |\La_{J'}-\La_{J'}(t_0)|}{|\La
_{J'}|^{1/2}|-\Theta_{J'}(t_0)|^{1/2}}\Psi \biggl(
\frac{u}{\sigma
_T} \biggr)\P \bigl\{Z_{J'}(t_0)\in
\R^{k'-k}_- \bigr\}
\\
&&\qquad\quad{} \times\P\bigl\{\bigl(X_{J'_1}(t_0),\ldots,X_{J'_{N-k'}}(t_0)
\bigr)\in\R ^{N-k'}_+|\nabla X_{|J'}(t_0)=0\bigr\}
\bigl(1+o(1)\bigr),\nonumber %
\end{eqnarray}
where $Z_{J'}(t_0)$ is a centered $(k'-k)$-dimensional Gaussian vector
with covariance matrix
$-(\tau_{ij})_{i,j\in\sigma(J')\setminus\sigma(J)}$. Plugging
(\ref
{Eq:unique max point 4})
and (\ref{Eq:unique max point 5}) into (\ref{Eq:unique max point 1}),
we obtain the asymptotic result.

(ii) $k=0$, say $J=\{t_0\}$. Note that $X(t_0)$ and $\nabla X(t_0)$
are independent, therefore,
%
\begin{equation}
\label{Eq:unique max point 2} %
\E\bigl\{M_u^E (J) \bigr\} =
\Psi \biggl(\frac{u}{\sigma_T} \biggr)\P\bigl\{\nabla X(t_0) \in
\R^N_+\bigr\}. %
\end{equation}
For $J'\in\partial_{k'} T$ with $\bar{J'}\cap\bar{J}\neq
\varnothing$,
then $\E\{M_u^E(J')\}$ is
given by (\ref{Eq:unique max point 5})
with $k=0$. Plugging (\ref{Eq:unique max point 2}) and (\ref{Eq:unique
max point 5}) into
(\ref{Eq:unique max point 1}), we obtain the asymptotic formula for the
excursion probability.
\end{remark}

\begin{example}[(Continued: The cosine field)]
\label{Example:unique maximum 2b} We consider the Gaussian field $X= \{X(t), t \in\R^2\}$ defined in
Example~\ref{Example:unique maximum 1b}.

(i) Let $T=[0, \pi]^2$. Then $\nu(t)$ attains its maximum 4 only at
the corner $t^*=(\pi, \pi)$,
where $\nabla\nu(t^*)=0$ so that the condition (\ref
{Condition:boundary max point}) is not satisfied.
Instead, we will use the result (ii) in Remark~\ref{Example:unique
maximum 2} with $J=\{t^*\}$ and
$k=0$. Let $J'=(0, \pi)\times\{\pi\}$, $J''=\{\pi\}\times(0, \pi)$.
Combining the results in Example~\ref{Example:unique maximum 1b} with (\ref{Eq:unique max point 2}) and
(\ref{Eq:unique max point 5}), and noting that $\La=\frac{1}{2}I_2$
implies $X_1(t)$ and $X_2(t)$
are independent for all $t$, we obtain
\begin{eqnarray*}
\E\bigl\{M_u^E (J)\bigr\}&=&
\tfrac{1}{4}\Psi(u/2),\qquad \E\bigl\{M_u^E (
\partial_2 T)\bigr\} = \tfrac{1}{2}\Psi(u/2) \bigl(1+o(1)\bigr),
\\
\E\bigl\{M_u^E \bigl(J'\bigr)\bigr\}&=&\E
\bigl\{M_u^E \bigl(J''\bigr)
\bigr\}=\tfrac{\sqrt{2}}{4}\Psi(u/2) \bigl(1+o(1)\bigr). %
\end{eqnarray*}
Summing these up, we have $\P\{\sup_{t\in T} X(t) \geq u \}
=[(3+2\sqrt
{2})/4]\Psi(u/2)(1+o(1))$.

(ii) Let $T=[0, 3\pi/2]\times[0, \pi]$. Then $\nu(t)$ attains its
maximum 4 only at the boundary
point $t^*=(\pi, \pi)$, where $\nu_2(t^*)=0$. Applying the result (i)
in Remark~\ref{Example:unique maximum 2} with $J=(0, 3\pi/2)\times\{\pi\}$ and
$k=1$, we obtain
\[
\E\bigl\{M_u^E (J)\bigr\}= \tfrac{\sqrt{2}}{2}\Psi(u/2),\qquad
\E\bigl\{M_u^E (\partial_2 T)\bigr\} =
\Psi(u/2) \bigl(1+o(1)\bigr),
\]
which implies that $\P\{\sup_{t\in T} X(t) \geq u \}=[(2+\sqrt
{2})/2]\Psi(u/2)(1+o(1))$.
\end{example}

\begin{remark} Note that we only provide the first-order approximation
for the examples in this section.
However, as shown in the theory of the approximations of integrals
[see, e.g., \citet{W01}], the integrals
in (\ref{Eq:excursion probability for special case}) and (\ref{Eq:MEC
approximation 2}) can be expanded
with more terms once the covariance function of the Gaussian field is
smooth enough. Hence, for the examples
above, higher-order approximation is available. Since the procedure is
similar and the computation
is tedious, we omit such argument here.
\end{remark}

\begin{appendix}\label{app}
\section*{Appendix}

This appendix contains proofs of Lemmas \ref{Lem:unifrom posdef}, \ref
{Lem:rho}
and some other auxiliary facts.

\begin{pf*}{Proof of Lemma~\ref{Lem:unifrom posdef}}
Let $M_{N\times N}$ be the set of all $N\times N$ matrices. Define a
mapping $\phi:
\R^N \times M_{N\times N} \rightarrow\R$ by $ (\xi, A) \mapsto\langle
\xi
, A\xi\rangle $.
Then $\phi$ is continuous. Since $ \La-\La(t)$ is positive definite,
we have $\phi(e, \La-\La(t))>0$ for all $t\in T$ and $e\in\mathbb
{S}^{N-1}$.
The conclusion of the lemma follows from this and the fact that $\{(e,
\La-\La(t))\dvtx t\in T, e\in\mathbb{S}^{N-1}\}$ is a compact
subset of $\R^N \times
M_{N\times N}$
and $\phi$ is continuous.
\end{pf*}

\begin{pf*}{Proof Lemma~\ref{Lem:rho}} We only prove the first
case, since the
second case follows from the first one. By elementary computation on
the joint
density of $\xi_1(t)$ and $\xi_2(s)$, we obtain
\begin{eqnarray*}
&&\sup_{t\in T_1, s\in T_2} \E \bigl\{\bigl(1+\bigl|
\xi_1(t)\bigr|^{N_1} + \bigl|\xi _2(s)\bigr|^{N_2}
\bigr) \mathbh{1}_{\{\xi_1(t) \geq u, \xi_2(s)<0\}} \bigr\}
\\
&&\qquad \leq\frac{1}{2\pi\underline{\sigma}_1\underline{\sigma}_2
(1-\overline{\rho}^2)^{1/2}} \int_u^\infty\exp \biggl
\{- \frac
{x_1^2}{2\overline{\sigma}_1^2} \biggr\} \,dx_1
\\
&&\qquad\quad{}\times \int_{-\infty}^0 \bigl(1+|x_1|^{N_1}
+ |x_2|^{N_2}\bigr) \exp \biggl\{ -\frac{1}{2\overline{\sigma}_2^2(1-\underline{\rho}^2)}
\biggl(x_2-\frac
{\underline{\sigma}_2 \underline{\rho} x_1}{\overline{\sigma}_1} \biggr)^2 \biggr\}
\,dx_2
\\
&&\qquad =o \biggl( \exp \biggl\{- \frac{u^2}{2\overline{\sigma}_1^2}-\frac{
\underline{\sigma}_2^2\underline{\rho}^2 u^2}{2\overline{\sigma
}_2^2(1-\underline{\rho}^2) \overline{\sigma}_1^2 }+\ep
u^2 \biggr\} \biggr), %
\end{eqnarray*}
as $u \to\infty$, for any $\ep>0$.
\end{pf*}

A similar argument for proving Lemma~\ref{Lem:unifrom posdef} yields
the following result.

\begin{lemmaa} \label{Lem:eigenvalue}
Let $\{ A(t)=(a_{ij}(t))_{1\leq i,j\leq N}\dvtx t\in T\}$ be a family of
positive definite
matrices such that all elements $a_{ij}(\cdot)$ are continuous. Denote
by $\underline{x}$
and $\overline{x}$ the infimum and supremum of the eigenvalues of
$A(t)$ over $t\in T$,
respectively, then $0<\underline{x}\leq\overline{x}<\infty$.
\end{lemmaa}

The following two formulas state the results on the Laplace
approximation method.
Lemma~\ref{Lem:Laplace method 1} can be found in many books on the
approximations
of integrals; here we refer to \citet{W01}. Lemma~\ref{Lem:Laplace
method 2} can
be derived by following similar arguments in the proof of the Laplace method
for the case of boundary point in \citet{W01}.

\begin{lemmaa}[(Laplace method for
interior point)]\label{Lem:Laplace method 1}
Let $t_0$ be an interior point of $T$. Suppose the following conditions
hold: \textup{(i)}
$g(t) \in C(T)$ and $g(t_0) \neq0$; \textup{(ii)} $h(t) \in C^2(T)$ and attains
its unique
minimum at $t_0$; and \textup{(iii)} $\nabla^2 h(t_0)$ is positive definite.
Then as $u \to\infty$,
\[
\int_T g(t) e^{-uh(t)} \,dt =\frac{(2\pi)^{N/2}}{u^{N/2} (\operatorname{ det}
\nabla^2
h(t_0))^{1/2}}
g(t_0) e^{-uh(t_0)} \bigl(1+ o(1)\bigr).
\]
\end{lemmaa}

\begin{lemmaa}[(Laplace method for
boundary point)]\label{Lem:Laplace method 2}
Let $t_0\in J \in\partial_k T$ with $0\leq k\leq N-1$. Suppose that
conditions \textup{(i)},
\textup{(ii)} and \textup{(iii)} in Lemma~\ref{Lem:Laplace method 1} hold, and additionally
$\nabla h(t_0)=0$. Then as $u \to\infty$,
\[
\int_T g(t) e^{-uh(t)} \,dt =
\frac{(2\pi)^{N/2} \P\{Z_J(t_0)\in
(-E(J))\}
}{u^{N/2}
(\operatorname{ det} \nabla^2 h(t_0))^{1/2}}g(t_0) e^{-uh(t_0)} \bigl(1+ o(1)\bigr),
\]
where $Z_J(t_0)$ is a centered $(N-k)$-dimensional Gaussian vector with
covariance matrix
$(h_{ij}(t_0))_{J_1\leq i,j\leq J_{N-k}}$, $-E(J)=\{x\in\R^N\dvtx
-x\in
E(J)\}$, and the
definitions of $J_1, \ldots, J_{N-k}$ and $E(J)$ are in (\ref{Def:EJ}).
\end{lemmaa}
\end{appendix}

\section*{Acknowledgements} We thank the anonymous referees and
the Editor for
their insightful comments which have led to several improvements of
this manuscript.

%
%

%




\printaddresses
\end{document}